\nonstopmode 
\def\hstrut#1{\rule{#1}{0cm}}

\def\strut{\hstrut{0cm}}

\def\mip{\par\medskip  }

\def\DS{\displaystyle}

\def\SSS{\scriptscriptstyle} 

\def\stB{\textstyle}
\def\stC{\scriptstyle}
\def\stD{\scriptscriptstyle}

\def\0{\mbox{\large $\varnothing$}}
\def\1#1{{ \f1{#1} }}
\def\2#1{{ \f{#1}2 }}
\def\3#1{{ \f{#1}3 }}
\def\4#1{{ \f{#1}4 }}
\def\5#1{\;\,\mbox{#1}\,\;}  % for text in maths mode
\def\6{\partial}
\def\7#1{{\mbox{\tiny \rm #1}}}  % klein wie sieben Zwerge
\def\8{\infty}
\def\9#1{{\SSS\rm{#1}}}

\def\q{\quad}
\def\bq{\!\!\!\!\!\!}

\def\f{\frac}

\def\x{\times}

\def\nn{\nonumber}
\def\sr{\stackrel} % {oben}{unten}
\def\cd{\cdot}
\def\ub{\underbrace} % Untertext mit _ : danach
\def\wt{\widetilde}
\def\mc{\multicolumn} % \mc{cols}{pos}{text}
\def\hs{\hspace*}

\def\lgl{\langle} \def\rgl{\rangle}
\def\opp{{\7{op}}}
\def\llongrightarrow{-\!\!\!-\!\!\!\longrightarrow}
\def\mto{\mapsto}

\def\aa{\alpha} \def\ba{\beta} \def\ga{\gamma} \def\da{\delta}

\def\Ga{\Gamma}

\def\phi{\varphi}

\def\rho{\varrho}

\def\Mbb{\mathbb}
  \def\C{{\Mbb C}}

 \def\N{{\Mbb N}} 
  \def\R{{\Mbb R}}
\def\S{{\Mbb S}}  
  
 \def\Z{{\Mbb Z}}
\def\cal{\mathcal}
\def\cA{{\cal A}}  \def\cC{{\cal C}}
\def\cD{{\cal D}} \def\cE{{\cal E}} \def\cF{{\cal F}}

\def\bA{{\bf A}}

\def\bP{{\bf P}}

\def\eq{      \mathrel {\;        =   \;}}

\def\defeq{   \mathrel {\;\colon\!{=} \;}}

\def\eig{    \mathop{\rm eig  } }
\def\ker{    \mathop{\rm ker  } }

\def\Lq{{\sqsubset}}
\def\Rq{{\sqsupset}}

\def\id{     \mathord{\rm id}  }

\newenvironment{proof} {{\bf Proof }}{\strut\hfill\epsy\bigskip} \def\epsy{$\blacklozenge$}
\newtheorem{theorem}{Theorem}

\newtheorem{lemma}[theorem]{Lemma}
\newfont{\fiverm}{cmr5 scaled 1000}
\newfont{\sixrm}{cmr6 scaled 1000}
\newfont{\sevenrm}{cmr7 scaled 1000}
\newfont{\eightrm}{cmr8 scaled 1000}
\newfont{\ninerm}{cmr9 scaled 1000}
\newfont{\tenrm}{cmr10 scaled 1000}

\input  prepictex.tex
\input  pictex.tex
\input  postpictex.tex
\def\SC #1 #2 {\setcoordinatesystem units <1mm,1mm> point at {-#1} {-#2} }

\documentclass[11pt,fleqn]{article}
\usepackage[cp850]{inputenc}
\usepackage{amssymb}
\begin{document}

\title{Orthogonal Polynomials appearing in SIE grid representations}

\author{Stefan Hilger \\ 
Katholische Universit\"{a}t Eichst\"{a}tt--Ingolstadt } 
\maketitle

\begin{abstract}
We show in this article how orthogonal polynomials appear in certain representations of grid shaped quivers.
After a short introduction into the general notion of quivers and their representations by linear operators
we define the notion of an SIE quiver representation: All intrinsic endomorphisms arising from circuits in the
quiver act as scalar multipliers.
We then present several lemmas that ensure this SIE property of a quiver representation.
Ladder commutator conditions and certain diagram commutativity ``up to scalar multiples'' play a central role.
The theory will then be applied to three examples. Extensive calculations shows how Associated Laguerre, Legendre--Gegenbauer
polynomials and binomial distributions fit into the framework of grid shaped SIE quivers.
One can see, that this algebraic point of view is foundational for orthogonal polynomials and special functions.
\end{abstract}

\mip
\mip

\newpage
\section{Introduction}
We show in this article how orthogonal polynomials appear in certain representations of grid shaped quivers.
After a short introduction into the general notion of quivers and their representations by linear operators
between vector spaces we define the notion of an SIE quiver representation: All intrinsic endomorphisms arising from circuits in the
quiver act as scalar multiples.
We then present several theorems that ensure this SIE property of a quiver representation.
Ladder commutator conditions and certain diagram commutativity ``up to scalar multiples'' play a central role.
The theory will then be applied to three examples. Extensive calculations show how Laguerre, Legendre--Gegenbauer polynomials and
the discrete binomial distribution fit into the framework of grid shaped SIE quivers.
This point of view onto orthogonal polynomials seems not to be known, see \cite{Dunkl/Xu}, \cite{Koelink/VanAssche}, \cite{Koekoek/Swarttouw}, e.g.
Note that we study operator algebraic aspects of orthogonal polynomials, we do not pursue orthogonality in the literal sense
or functional analytic or numerical approaches.
\mip
An  enhancement of the theory including further examples of orthogonal polynomials, $q$--orthogonal polynomials connected by operators of
the Weyl algebra, $h$--Weyl algebra, $q$--Weyl algebra or $(q,h)$--Weyl algebra, see \cite{Hilger/Filipuk}, is not included here.

\section{Grids}

In this article we will study so called grids of operators.
Let $\cA$ be a fixed associative unital algebra over $\C$ with a representation on a vector space $V$.

A grid is a directed graph with vertices $\bullet_{nm}$, indexed by pairs $(n,m) \in \Z\x\Z$.
Between two neighbouring vertices there is a pair of arrows, one in each direction.
There are horizontal, vertical and diagonal arrows.
The local situation at the vertex $\bullet_{nm}$ is shown in the diagram

\begin{equation}
\def\plotSA #1 #2 #3 #4 /{ \arrow <1mm> [.25,.75] from #1 #2 to #3 #4 }
\def\SC #1 #2 {\setcoordinatesystem units <.5mm,.5mm> point at {-#1} {-#2} }
\label{grid quiver local}
\beginpicture
\SC  00  00 \setplotarea x from -90 to +70, y from -60 to +60

\SC -60 -60 \put { $\bullet_{n-1,m-1}$ } [] at 00 00
\SC -60  00 \put { $\bullet_{n-1,m  }$ } [] at 00 00
\SC -60  60 \put { $\bullet_{n-1,m+1}$ } [] at 00 00
\SC  00 -60 \put { $\bullet_{n  ,m-1}$ } [] at 00 00
\SC  00  00 \put { $\bullet_{n  ,m  }$ } [] at 00 00
\SC  00  60 \put { $\bullet_{n  ,m+1}$ } [] at 00 00
\SC +60 -60 \put { $\bullet_{n+1,m-1}$ } [] at 00 00
\SC  60  00 \put { $\bullet_{n+1,m  }$ } [] at 00 00
\SC +60 +60 \put { $\bullet_{n+1,m+1}$ } [] at 00 00

% East
\SC  00  00 \plotSA +20   +01   +40   +01   / \put {\tiny $a^+_{n+1,m  } $} [b ] at +30  02
\SC  00  60 \plotSA +20   +01   +40   +01   / \put {\tiny $a^+_{n+1,m+1} $} [b ] at +30  02
\SC -60  60 \plotSA +20   +01   +40   +01   / \put {\tiny $a^+_{n  ,m+1} $} [b ] at +30  02
\SC -60  00 \plotSA +20   +01   +40   +01   / \put {\tiny $a^+_{n  ,m  } $} [b ] at +30  02
\SC -60 -60 \plotSA +20   +01   +40   +01   / \put {\tiny $a^+_{n  ,m-1} $} [b ] at +30  02
\SC  00 -60 \plotSA +20   +01   +40   +01   / \put {\tiny $a^+_{n+1,m-1} $} [b ] at +30  02

% West
\SC  00  00 \plotSA -20   -01   -40   -01   / \put {\tiny $a^-_{n  ,m  } $} [t ] at -30 -02
\SC  60  00 \plotSA -20   -01   -40   -01   / \put {\tiny $a^-_{n+1,m  } $} [t ] at -30 -02
\SC  60  60 \plotSA -20   -01   -40   -01   / \put {\tiny $a^-_{n+1,m+1} $} [t ] at -30 -02
\SC  00  60 \plotSA -20   -01   -40   -01   / \put {\tiny $a^-_{n  ,m+1} $} [t ] at -30 -02
\SC  00 -60 \plotSA -20   -01   -40   -01   / \put {\tiny $a^-_{n  ,m-1} $} [t ] at -30 -02
\SC  60 -60 \plotSA -20   -01   -40   -01   / \put {\tiny $a^-_{n+1,m-1} $} [t ] at -30 -02

% North
\SC  00  00 \plotSA -01   +20   -01   +40   / \put {\tiny $b^+_{n  ,m+1} $} [r ] at -02 +30
\SC  60  00 \plotSA -01   +20   -01   +40   / \put {\tiny $b^+_{n+1,m+1} $} [r ] at -02 +30
\SC -60  00 \plotSA -01   +20   -01   +40   / \put {\tiny $b^+_{n-1,m+1} $} [r ] at -02 +30
\SC -60 -60 \plotSA -01   +20   -01   +40   / \put {\tiny $b^+_{n-1,m  } $} [r ] at -02 +30
\SC  00 -60 \plotSA -01   +20   -01   +40   / \put {\tiny $b^+_{n  ,m  } $} [r ] at -02 +30
\SC  60 -60 \plotSA -01   +20   -01   +40   / \put {\tiny $b^+_{n+1,m  } $} [r ] at -02 +30

% South
\SC  00  00 \plotSA +01   -20   +01   -40   / \put {\tiny $b^-_{n  ,m  } $} [l ] at +02 -30
\SC  60  00 \plotSA +01   -20   +01   -40   / \put {\tiny $b^-_{n+1,m  } $} [l ] at +02 -30
\SC  60  60 \plotSA +01   -20   +01   -40   / \put {\tiny $b^-_{n+1,m+1} $} [l ] at +02 -30
\SC  00  60 \plotSA +01   -20   +01   -40   / \put {\tiny $b^-_{n  ,m+1} $} [l ] at +02 -30
\SC -60  60 \plotSA +01   -20   +01   -40   / \put {\tiny $b^-_{n-1,m+1} $} [l ] at +02 -30
\SC -60  00 \plotSA +01   -20   +01   -40   / \put {\tiny $b^-_{n-1,m  } $} [l ] at +02 -30

% Northeast
\SC  00  00 \plotSA +19.5 +20.5 +39.5 +40.5 / \put {\tiny $p^+_{n+1,m+1} $} [rt] at +33  40
\SC -60 -60 \plotSA +19.5 +20.5 +39.5 +40.5 / \put {\tiny $p^+_{n  ,m  } $} [rt] at +33  40

% Northwest
\SC  00  00 \plotSA -19.5 +20.5 -39.5 +40.5 / \put {\tiny $q^+_{n-1,m+1} $} [lt] <-1mm,+1mm> at -34  40
\SC  60 -60 \plotSA -19.5 +20.5 -39.5 +40.5 / \put {\tiny $q^+_{n  ,m  } $} [lt] <-1mm,+1mm> at -34  40

% Southwest
\SC  00  00 \plotSA -19.5 -20.5 -39.5 -40.5 / \put {\tiny $p^-_{n  ,m  } $} [lb] at -34 -40
\SC  60  60 \plotSA -19.5 -20.5 -39.5 -40.5 / \put {\tiny $p^-_{n+1,m+1} $} [lb] at -34 -40

% Southeast
\SC  00  00 \plotSA +19.5 -20.5 +39.5 -40.5 / \put {\tiny $q^-_{n  ,m  } $} [rb] <+1mm,-1mm> at +34 -40
\SC -60  60 \plotSA +19.5 -20.5 +39.5 -40.5 / \put {\tiny $q^-_{n-1,m+1} $} [rb] <+1mm,-1mm> at +34 -40
\endpicture
\end{equation}

One can see that we have adopted the following general convention for indices of arrows:
An arrow with superscript $+$ has the index of its target vertex.
An arrow with superscript $-$ has the index of its source vertex.

\section{Quivers}

We consider a grid as a special quiver.
Instead of a rigoros definition we describe the notion of a quiver as follows:
A \emph{quiver} is a directed graph $G$. It consists of a set $\bP = \{ \bullet_i | i \in I \}$ of points and a
set $\bA = \{ a_j | j \in J \}$, of arrows connecting the vertices. More accurately, to each arrow $a$ one
can assign a source vertex $s(a)$ and a target vertex $t(a)$ in $\bP$.
The quivers which we are dealing with are infinite, they will have circuits, multiple arrows are allowed.
\mip

A path in $G$ is a sequence $(a_m\, a_{m-1} \cdots a_2\, a_1)$ of consecutive arrows in $\bA$, that is
\[ t(a_{i}) = s(a_{i+1}) \q\q \5{for all} i = 1,\dots,m-1.
\]
We can extend the source and target maps to the set of paths by
\[ s(a_m \cdots a_1) = s(a_1),\q\q t(a_m\cdots a_1) = t(a_m).
\]
If $ a = (a_m \cdots a_1), \wt a = (\wt a_{\wt m} \cdots  \wt a_1)$ are two paths with
$t(a) = s(\wt a)$ we define their product by concatenation or arrows
\[ (\wt a_{\wt m} \cdots \wt a_1) \cd (a_m \cdots a_1) =
                     (\wt a_{\wt m} \cdots \wt a_1 \, a_m \cdots a_1).
\]
In this article we will study grids or quivers that only have double arrows, as shown in the above diagram.
In mathematical terms this means that
that we can assign to each arrow $a$ a unique arrow $a^\opp$ between the same vertices pointing in the opposite direction, that is
\[ t(a^\opp) = s(a), \q\q s(a^\opp) = t(a).
\]
Then for every path $a = (a_m \cdots a_1)$ there is an \emph{opposite} path,
\[  (a_m \cdots a_1)^\opp = (a_1^\opp \cdots a_m^\opp).
\]

A path $a$ is called a \emph{circuit} if $s(a) = t(a)$.
A circuit $a$ is called \emph{narrow}, if there is a path $b$, such that $a \eq b^\opp b$.
This means that the circuit the path $b$ and then returns along the same path in opposite direction.
Circuits that are not narrow will be called \emph{wide}.
A narrow circuit $a$ is called \emph{short}, if there is an arrow $b$, such that $a\eq b^\opp b$.

\section{Representations of Quivers}
Like every directed graph, a quiver or grid can be considered as a category with the vertices as objects and the arrows as morphisms.
A \emph{representation} $\rho$ of this quiver is just a functor from this category to the category of
(complex) vector spaces. In other words it is an assignment
\[ \rho:
\left\{
\begin{array}{rcl}
\bullet_i & \mto & V_i \\
  a_j     & \mto & A_j
\end{array}
\right.
\]
of vertices $\bullet_i$ to vector spaces $V_i$ and of arrows $a_j$ to linear operators $A_j$ acting between these vector spaces
in the obvious  way:
\[ A_j: V_{s(a_j)} \to V_{t(a_j)}.
\]
So, a representation of a quiver is just a diagram of vector spaces and operators such that the quiver carries the information  about the diagram structure.
This situation is shown in the following diagram, the grid shape serves as an example again.

\begin{equation}
\def\plotSA #1 #2 #3 #4 /{ \arrow <1mm> [.25,.75] from #1 #2 to #3 #4 }
\def\SC #1 #2 {\setcoordinatesystem units <.5mm,.5mm> point at {-#1} {-#2} }
\label{abstract general grid}
\beginpicture
\SC  00  00
\setplotarea x from -90 to +70, y from -70 to +70

\SC -60 -60 \put { $V_{n-1,m-1}$ } [] at 00 00
\SC -60  00 \put { $V_{n-1,m  }$ } [] at 00 00
\SC -60  60 \put { $V_{n-1,m+1}$ } [] at 00 00
\SC  00 -60 \put { $V_{n  ,m-1}$ } [] at 00 00
\SC  00  00 \put { $V_{n  ,m  }$ } [] at 00 00
\SC  00  60 \put { $V_{n  ,m+1}$ } [] at 00 00
\SC +60 -60 \put { $V_{n+1,m-1}$ } [] at 00 00
\SC  60  00 \put { $V_{n+1,m  }$ } [] at 00 00
\SC +60 +60 \put { $V_{n+1,m+1}$ } [] at 00 00

% East
\SC  00  00 \plotSA +20   +01   +40   +01   / \put {\tiny $A^+_{n+1,m  } $} [b ] at +30  02
\SC  00  60 \plotSA +20   +01   +40   +01   / \put {\tiny $A^+_{n+1,m+1} $} [b ] at +30  02
\SC -60  60 \plotSA +20   +01   +40   +01   / \put {\tiny $A^+_{n  ,m+1} $} [b ] at +30  02
\SC -60  00 \plotSA +20   +01   +40   +01   / \put {\tiny $A^+_{n  ,m  } $} [b ] at +30  02
\SC -60 -60 \plotSA +20   +01   +40   +01   / \put {\tiny $A^+_{n  ,m-1} $} [b ] at +30  02
\SC  00 -60 \plotSA +20   +01   +40   +01   / \put {\tiny $A^+_{n+1,m-1} $} [b ] at +30  02

% West
\SC  00  00 \plotSA -20   -01   -40   -01   / \put {\tiny $A^-_{n  ,m  } $} [t ] at -30 -02
\SC  60  00 \plotSA -20   -01   -40   -01   / \put {\tiny $A^-_{n+1,m  } $} [t ] at -30 -02
\SC  60  60 \plotSA -20   -01   -40   -01   / \put {\tiny $A^-_{n+1,m+1} $} [t ] at -30 -02
\SC  00  60 \plotSA -20   -01   -40   -01   / \put {\tiny $A^-_{n  ,m+1} $} [t ] at -30 -02
\SC  00 -60 \plotSA -20   -01   -40   -01   / \put {\tiny $A^-_{n  ,m-1} $} [t ] at -30 -02
\SC  60 -60 \plotSA -20   -01   -40   -01   / \put {\tiny $A^-_{n+1,m-1} $} [t ] at -30 -02

% North
\SC  00  00 \plotSA -01   +20   -01   +40   / \put {\tiny $B^+_{n  ,m+1} $} [r ] at -02 +30
\SC  60  00 \plotSA -01   +20   -01   +40   / \put {\tiny $B^+_{n+1,m+1} $} [r ] at -02 +30
\SC -60  00 \plotSA -01   +20   -01   +40   / \put {\tiny $B^+_{n-1,m+1} $} [r ] at -02 +30
\SC -60 -60 \plotSA -01   +20   -01   +40   / \put {\tiny $B^+_{n-1,m  } $} [r ] at -02 +30
\SC  00 -60 \plotSA -01   +20   -01   +40   / \put {\tiny $B^+_{n  ,m  } $} [r ] at -02 +30
\SC  60 -60 \plotSA -01   +20   -01   +40   / \put {\tiny $B^+_{n+1,m  } $} [r ] at -02 +30

% South
\SC  00  00 \plotSA +01   -20   +01   -40   / \put {\tiny $B^-_{n  ,m  } $} [l ] at +02 -30
\SC  60  00 \plotSA +01   -20   +01   -40   / \put {\tiny $B^-_{n+1,m  } $} [l ] at +02 -30
\SC  60  60 \plotSA +01   -20   +01   -40   / \put {\tiny $B^-_{n+1,m+1} $} [l ] at +02 -30
\SC  00  60 \plotSA +01   -20   +01   -40   / \put {\tiny $B^-_{n  ,m+1} $} [l ] at +02 -30
\SC -60  60 \plotSA +01   -20   +01   -40   / \put {\tiny $B^-_{n-1,m+1} $} [l ] at +02 -30
\SC -60  00 \plotSA +01   -20   +01   -40   / \put {\tiny $B^-_{n-1,m  } $} [l ] at +02 -30

% Northeast
\SC  00  00 \plotSA +19.5 +20.5 +39.5 +40.5 / \put {\tiny $P^+_{n+1,m+1} $} [rt] at +33  40
\SC -60 -60 \plotSA +19.5 +20.5 +39.5 +40.5 / \put {\tiny $P^+_{n  ,m  } $} [rt] at +33  40

% Northwest
\SC  00  00 \plotSA -19.5 +20.5 -39.5 +40.5 / \put {\tiny $Q^+_{n-1,m+1} $} [lt] <-1mm,+1mm> at -34  40
\SC  60 -60 \plotSA -19.5 +20.5 -39.5 +40.5 / \put {\tiny $Q^+_{n  ,m  } $} [lt] <-1mm,+1mm> at -34  40

% Southwest
\SC  00  00 \plotSA -19.5 -20.5 -39.5 -40.5 / \put {\tiny $P^-_{n  ,m  } $} [lb] at -34 -40
\SC  60  60 \plotSA -19.5 -20.5 -39.5 -40.5 / \put {\tiny $P^-_{n+1,m+1} $} [lb] at -34 -40

% Southeast
\SC  00  00 \plotSA +19.5 -20.5 +39.5 -40.5 / \put {\tiny $Q^-_{n  ,m  } $} [rb] <+1mm,-1mm> at +34 -40
\SC -60  60 \plotSA +19.5 -20.5 +39.5 -40.5 / \put {\tiny $Q^-_{n-1,m+1} $} [rb] <+1mm,-1mm> at +34 -40
\endpicture
\end{equation}

A \emph{homomorphism} of two representations $(V_i,A_j)$ and $(W_i,B_j)$ is just a collection of linear operators $(\phi_i:V_i \to W_i)_{i\in I}$ that commute with the arrow operators $A_j$ in the right way: For any arrow $a_j$ in the quiver we have
\[ B_j \, \phi_{s(j)} \eq \phi_{t(j)} \, A_j \q: \q V_{s(j)} \to W_{t(j)}.
\]
To each path $a= (a_m \cdots a_1)$ we can assign the operator product $A = A_m \cdots A_1$.
The path product corresponds to the operator product. An operator assigned to a circuit is a
vector space endomorphism. We will call it a \emph{circuit endomorphism}.

\mip
%\section{SIE Representations}
Now we consider a fixed representation $\rho$ of a quiver.
We say that a circuit endomorphism $A_j:V_{s(j)} \to V_{s(j)}$ \emph{acts as a scalar} or it is a \emph{\textbf{S}calar \textbf{I}ntrinsic \textbf{E}ndomorphism (SIE)}, if it is a scalar multiple of the identity mapping on $V_{s(j)}$
\[ A_j \eq \aa \cd \id_{V_{s(j)}} \eq \aa \q\q \5{for some} \aa\in\C.
\]
Here and from now on we identify a scalar multiple of the identity operator with the scalar itself.
\mip
The quiver representation $\rho$ will be called an \emph{SIE representation}, if all circuit endomorphisms arising in it are SIE.
The following lemmas collect some simple linear algebraic observations about SIEs.

\begin{lemma}[Narrow circuit endomorphisms]
\label{scalars equal}
Let $a$ be a path in the quiver and $A = \rho(a), A^\opp = \rho(a^{\opp})$ the corresponding operators in both directions
\begin{eqnarray*}
\beginpicture
\SC  00 00 \setplotarea x from -45 to +45, y from -03 to +03
\SC -14 00 \put {$V_{s(a)}$} [] at 00 00
\SC  00 00 \arrow <1mm> [.25,.75] from -05 01 to +05  01 \put {\tiny $A     $} [b] at 00  02
           \arrow <1mm> [.25,.75] from +05 00 to -05  00 \put {\tiny $A^\opp$} [t] at 00 -01
\SC +14 00 \put {$V_{t(a)}.$} [] at 00 00
\endpicture
\end{eqnarray*}

\begin{itemize}
\item[(i)] If the two narrow circuit endomorphisms $A^\opp A$ and $A A^\opp$ act as scalars
\begin{eqnarray*}
&& A^\opp A \eq \aa,  \q\q A A^\opp \eq \wt\aa,
\end{eqnarray*}
then these scalars are equal: $\aa = \wt\aa$.

\mip The example
\beginpicture
\SC  00 00 \setplotarea x from -13 to +13, y from +02 to 05
\SC -10 00 \put {$\R$} [] at 00 01
\SC  00 00 \arrow <1mm> [.25,.75] from -05 02 to +05  02 \put {\tiny $ \id \x 0 $} [b] at 00  03
           \arrow <1mm> [.25,.75] from +05 01 to -05  01 \put {\tiny $ \mathop{\rm pr}_1  $} [t] at 00  00
\SC +10 00 \put {$\R^2$} [] at 00 01
\endpicture
shows that the two conditions do not imply each other.
\mip
\item[(iii)]
Let $\aa\in\C\setminus\{ 0 \}$ be a nonzero number. The following statements about the corresponding circuit endomorphisms are equivalent
\begin{itemize}
\item[] $A^\opp A      = \aa \; \id_V $ and $A$ surjective.
\item[] $A^\opp A      = \aa \; \id_V $ and $A^\opp$ injective.
\item[] $A      A^\opp = \aa \; \id_W $ and $A^\opp$ surjective.
\item[] $A      A^\opp = \aa \; \id_W $ and $A$ injective.
\item[] $A^\opp A      = \aa \; \id_V $ and $A A^\opp = \aa \; \id_W $. $V$ and $W$ are isomorphic.
\end{itemize}
\item[(iv)] If all short circuit endomorphisms act as scalars then all circuit endomorphisms act as scalars.
\end{itemize}
\end{lemma}

\begin{proof}
The proof of (iii) is simply based on the fact that $BC = \aa \cd \id$ implies $B$ injective and $C$ surjective.
\mip
We are going to show (iv):
For a path $a = (a_m \cdots a_1)$ of arrows and its operator counterpart $A = A_m \cdots A_1$ we have the identity
\begin{eqnarray*}
A^\opp A &=& A_1^\opp \cdots \ub{ A_{m-1}^\opp \ub{A_m^\opp A_m} A_{m-1} } \cdots A_1 \eq \aa_m \cdots \aa_m
\end{eqnarray*}
Inductively --- starting with the innermost circuit operator $A_m^\opp A_m$ --- all the circuits act as scalars and commute with the outer operators.
\end{proof}

\begin{lemma}[Wide circuit endomorphisms]
\label{Wide circuit endomorphisms}
Consider the following part of a representation $(U_{nm})$ of a grid.
\begin{equation}
\def\plotSA #1 #2 #3 #4 /{ \arrow <1mm> [.25,.75] from #1 #2 to #3 #4 }
\def\SC #1 #2 {\setcoordinatesystem units <.5mm,.5mm> point at {-#1} {-#2} }
\label{}
\beginpicture
\SC  00  00 \setplotarea x from -10 to +70, y from -10 to +70
\SC  00  00 \put { $V_{n  ,m  }$ } [] at 00 00
\SC  00  60 \put { $V_{n  ,m+1}$ } [] at 00 00
\SC  60  00 \put { $V_{n+1,m  }$ } [] at 00 00
\SC +60 +60 \put { $V_{n+1,m+1}$ } [] at 00 00

% East
\SC  00  00 \plotSA +20   +01   +40   +01   / \put {\tiny $A^+_{n+1,m  } $} [b ] at +30  02
\SC  00  60 \plotSA +20   +01   +40   +01   / \put {\tiny $A^+_{n+1,m+1} $} [b ] at +30  02

% West
\SC  60  00 \plotSA -20   -01   -40   -01   / \put {\tiny $A^-_{n+1,m  } $} [t ] at -30 -02
\SC  60  60 \plotSA -20   -01   -40   -01   / \put {\tiny $A^-_{n+1,m+1} $} [t ] at -30 -02

% North
\SC  00  00 \plotSA -01   +20   -01   +40   / \put {\tiny $B^+_{n  ,m+1} $} [r ] at -02 +30
\SC  60  00 \plotSA -01   +20   -01   +40   / \put {\tiny $B^+_{n+1,m+1} $} [r ] at -02 +30

% South
\SC  60  60 \plotSA +01   -20   +01   -40   / \put {\tiny $B^-_{n+1,m+1} $} [l ] at +02 -30
\SC  00  60 \plotSA +01   -20   +01   -40   / \put {\tiny $B^-_{n  ,m+1} $} [l ] at +02 -30
\endpicture
\end{equation}
Assume that the eight narrow circuit endomorphisms act as scalars as follows:
\begin{eqnarray}
\bq\bq\bq 
\begin{array}{r @{\;\,} c @{\;\,} l @{\q\q\q} r @{\;\,} c @{\;\,} l }
\DS A^-_{n+1,m  } A^+_{n+1,m  } & = & \DS \aa_{n+1,m  } &
\DS A^+_{n+1,m  } A^-_{n+1,m  } & = & \DS \aa_{n+1,m  } \\[1.5ex]
\DS B^-_{n+1,m+1} B^+_{n+1,m+1} & = & \DS \ba_{n+1,m+1} &
\DS B^+_{n+1,m+1} B^-_{n+1,m+1} & = & \DS \ba_{n+1,m+1} \\[1.5ex]
\DS A^+_{n+1,m+1} A^-_{n+1,m+1} & = & \DS \aa_{n+1,m+1} &
\DS A^-_{n+1,m+1} A^+_{n+1,m+1} & = & \DS \aa_{n+1,m+1} \\[1.5ex]
\DS B^+_{n  ,m+1} B^-_{n  ,m+1} & = & \DS \ba_{n  ,m+1} &
\DS B^-_{n  ,m+1} B^+_{n  ,m+1} & = & \DS \ba_{n  ,m+1}
\end{array}
\label{eight narrow circuit conditions}
\end{eqnarray}
Let $\ga_{nm}, \, \da_{nm}$ be two more numbers such that their product equals the product of the four numbers
\begin{eqnarray}
\ba_{n  ,m+1} \cd \aa_{n+1,m+1} \cd \ba_{n+1,m+1}  \cd \aa_{n+1,m  } &=& \ga_{nm} \cd \da_{nm}.
\label{wide circuit condition}
\end{eqnarray}

Then the following sixteen equations are equivalent.
\begin{small}
\begin{eqnarray*}
\bq\bq\bq
\begin{array}[c]{ r @{\;} l @{\;}c@{\;} l @{\;} l @{\;}c@{\;} l }
\circlearrowleft: & V_{n  ,m  }&\to&V_{n  ,m  } &\q B^-_{n  ,m+1} \cd A^-_{n+1,m+1} \cd B^+_{n+1,m+1} \cd A^+_{n+1,m  } &\eq& \ga_{nm} \\
\circlearrowleft: & V_{n+1,m  }&\to&V_{n+1,m  } &\q A^+_{n+1,m  } \cd B^-_{n  ,m+1} \cd A^-_{n+1,m+1} \cd B^+_{n+1,m+1} &\eq& \ga_{nm} \\
\circlearrowleft: & V_{n+1,m+1}&\to&V_{n+1,m+1} &\q B^+_{n+1,m+1} \cd A^+_{n+1,m  } \cd B^-_{n  ,m+1} \cd A^-_{n+1,m+1} &\eq& \ga_{nm} \\
\circlearrowleft: & V_{n  ,m+1}&\to&V_{n  ,m+1} &\q A^-_{n+1,m+1} \cd B^+_{n+1,m+1} \cd A^+_{n+1,m  } \cd B^-_{n  ,m+1} &\eq& \ga_{nm} \\
\circlearrowright:& V_{n  ,m  }&\to&V_{n  ,m  } &\q A^-_{n+1,m  } \cd B^-_{n+1,m+1} \cd A^+_{n+1,m+1} \cd B^+_{n  ,m+1} &\eq& \da_{nm} \\
\circlearrowright:& V_{n  ,m+1}&\to&V_{n  ,m+1} &\q B^+_{n  ,m+1} \cd A^-_{n+1,m  } \cd B^-_{n+1,m+1} \cd A^+_{n+1,m+1} &\eq& \da_{nm} \\
\circlearrowright:& V_{n+1,m+1}&\to&V_{n+1,m+1} &\q A^+_{n+1,m+1} \cd B^+_{n  ,m+1} \cd A^-_{n+1,m  } \cd B^-_{n+1,m+1} &\eq& \da_{nm} \\
\circlearrowright:& V_{n+1,m  }&\to&V_{n+1,m  } &\q B^-_{n+1,m+1} \cd A^+_{n+1,m+1} \cd B^+_{n  ,m+1} \cd A^-_{n+1,m  } &\eq& \da_{nm} \\[1.5ex]
\nearrow:& V_{n  ,m  }&\to&V_{n+1,m+1} &\mc{3}{l}{ \q A^+_{n+1,m+1} \cd B^+_{n  ,m+1}   \q\eq
                                          \f{\aa_{n+1,m+1} \cd \ba_{n  ,m+1} }{\ga_{nm}} \q   B^+_{n+1,m+1} \cd A^+_{n+1,m  } } \\[1.2ex]
\nwarrow:& V_{n+1,m  }&\to&V_{n  ,m+1} &\mc{3}{l}{ \q B^+_{n  ,m+1} \cd A^-_{n+1,m }   \q\q\eq
                                          \f{\aa_{n+1,m  } \cd \ba_{n  ,m+1} }{\ga_{nm}} \q\q A^-_{n+1,m+1} \cd B^+_{n+1,m+1} } \\[1.2ex]
\swarrow:& V_{n+1,m+1}&\to&V_{n  ,m  } &\mc{3}{l}{ \q A^-_{n+1,m  } \cd B^-_{n+1,m+1}   \q\eq
                                          \f{\aa_{n+1,m  } \cd \ba_{n+1,m+1} }{\ga_{nm}} \q   B^-_{n  ,m+1} \cd A^-_{n+1,m+1} } \\[1.2ex]
\searrow:& V_{n  ,m+1}&\to&V_{n+1,m  } &\mc{3}{l}{ \q B^-_{n+1,m+1} \cd A^+_{n+1,m+1}   \eq
                                          \f{\aa_{n+1,m+1} \cd \ba_{n+1,m+1} }{\ga_{nm}} \;   A^+_{n+1,m  } \cd B^-_{n  ,m+1} } \\[1.5ex]
\nearrow:& V_{n  ,m  }&\to&V_{n+1,m+1} &\mc{3}{l}{ \q B^+_{n+1,m+1} \cd A^+_{n+1,m  }   \q\eq
                                          \f{\aa_{n+1,m  } \cd \ba_{n+1,m+1} }{\da_{nm}} \q   A^+_{n+1,m+1} \cd B^+_{n  ,m+1} } \\[1.2ex]
\nwarrow:& V_{n+1,m  }&\to&V_{n  ,m+1} &\mc{3}{l}{ \q A^-_{n+1,m+1} \cd B^+_{n+1,m+1}   \eq
                                          \f{\aa_{n+1,m+1} \cd \ba_{n+1,m+1} }{\da_{nm}} \;   B^+_{n  ,m+1} \cd A^-_{n+1,m  } } \\[1.2ex]
\swarrow:& V_{n+1,m+1}&\to&V_{n  ,m  } &\mc{3}{l}{ \q B^-_{n  ,m+1} \cd A^-_{n+1,m+1}   \q\eq
                                          \f{\aa_{n+1,m+1} \cd \ba_{n  ,m+1} }{\da_{nm}} \q   A^-_{n+1,m  } \cd B^-_{n+1,m+1} } \\[1.2ex]
\searrow:& V_{n  ,m+1}&\to&V_{n+1,m  } &\mc{3}{l}{ \q A^+_{n+1,m  } \cd B^-_{n  ,m+1}   \q\q\eq
                                          \f{\aa_{n+1,m  } \cd \ba_{n  ,m+1} }{\da_{nm}} \q\q B^-_{n+1,m+1} \cd A^+_{n+1,m+1} }
\end{array}
\end{eqnarray*}
\end{small}
The first four equations state that the wide circuit endomorphisms along the square in counterclockwise direction act as the same scalar.
The next four equations contain the same statement for the clockwise direction.
The other eight equations describe the ``up to scalar commutativity'' of the diagram.
\end{lemma}

\begin{proof}
Just multiply the first equation with $A^-_{n+1,m}$ from the right and with $A^+_{n+1,m}$ from the left.
Then the second identity in (\ref{eight narrow circuit conditions}) proves the second equation.
\mip
Multiply the first equation with the operator of the fifth line from the left.
Then the left side of the identities in (\ref{eight narrow circuit conditions}) and the identity (\ref{wide circuit condition}) show the
fifth equation.
\mip
Multiply the first equation with the left operator of the ninth line from the left.
Then the fifth and seventh identity in (\ref{eight narrow circuit conditions}) show the ninth equation.
\mip
The equivalence of the ninth and tenth equation follows directly from (\ref{wide circuit condition}).
\mip
Cyclic commutation shows all the other implications.
\end{proof}

\section{Ladders}

We now consider a single (finite or infinite) path in the quiver together with its opposite path.
A representation of a path, called a \emph{ladder}, is illustrated by the diagram
\begin{equation}
\label{ladder}
\beginpicture
\SC  00 00 \setplotarea x from -65 to +65, y from -03 to 03
\SC -42 00 \put {$\dots$}   [] at 00 00
\SC +42 00 \put {$\dots$}   [] at 00 00
\SC -28 00 \put {$V_{n-1}$} [] at 00 00
\SC  00 00 \put {$V_{n  }$} [] at 00 00
\SC +28 00 \put {$V_{n+1}$} [] at 00 00
\SC -14 00 \arrow <1mm> [.25,.75] from -05 01 to +05  01 \put {\tiny $A^+_n$} [b] at 00  02
           \arrow <1mm> [.25,.75] from +05 00 to -05  00 \put {\tiny $A^-_n$} [t] at 00 -01
\SC +14 00 \arrow <1mm> [.25,.75] from -05 01 to +05  01 \put {\tiny $A^+_{n+1}$} [b] at 00  02
           \arrow <1mm> [.25,.75] from +05 00 to -05  00 \put {\tiny $A^-_{n+1}$} [t] at 00 -01
\endpicture
\end{equation}

Ladders appear abundantly in quantum mechanics and quantum field theory (see \cite{Dong}).
In \cite{Haran} they are used for studying Markov chains on trees.
A basic study of this notion and many examples can be found in \cite{HilgerCL}.
The following theorem from ladder theory is taken from \cite{HilgerCL},
the setting here is a little bit different, mainly due to other index conventions.
\mip
For a given general ladder we define the narrow circuit operators
\begin{eqnarray*}
A_n^\Lq &\defeq& A_{n  }^+ A_{n  }^-, \q\q\q
A_n^\Rq \defeq   A_{n+1}^- A_{n+1}^+.
\end{eqnarray*}

\begin{lemma}[Scalar commutators] \
\label{Lemma: Scalar commutators}
Assume that there is a sequence $(\aa_n)_{n\in\Z}$ of (real or complex) numbers, such that
\begin{eqnarray}
\label{commutator condition} A_n^\Rq - A_n^\Lq \eq A_{n+1}^- A_{n+1}^+ - A_{n}^+ A_{n}^- &=& \aa_{n+1} - \aa_{n} \q \5{for all} n \in \Z
\end{eqnarray}
Then the ladder
\begin{eqnarray}
\label{abs lad 3}
\beginpicture
\SC  00 00 \setplotarea x from -50 to +50, y from -03 to 03
\SC -42 00 \put {$ \dots $} [] at 00 00
\SC +42 00 \put {$ \dots $} [] at 00 00
\SC -28 00 \put {$\cE_{n-1}$} [] at 00 00
\SC  00 00 \put {$\cE_{n  }$} [] at 00 00
\SC +28 00 \put {$\cE_{n+1}$} [] at 00 00
\SC -14 00 \arrow <1mm> [.25,.75] from -05 01 to +05  01 \put {\tiny $A^+_ n   $} [b] at 00  02
           \arrow <1mm> [.25,.75] from +05 00 to -05  00 \put {\tiny $A^-_ n   $} [t] at 00 -01
\SC +14 00 \arrow <1mm> [.25,.75] from -05 01 to +05  01 \put {\tiny $A^+_{n+1}$} [b] at 00  02
           \arrow <1mm> [.25,.75] from +05 00 to -05  00 \put {\tiny $A^-_{n+1}$} [t] at 00 -01
\endpicture
\end{eqnarray}
with eigenspaces
\begin{eqnarray*}
\label{ladder eigenspaces}
\cE_n
&\defeq&
%\eig (A_{n+1}^- A_{n+1}^+  + A_{n}^+ A_{n}^-,\aa_{n+1} + \aa_{n}) \\
%&=& \nn
\eig ( A_n^\Lq , \aa_{n  } ) \eq \eig ( A_n^\Rq , \aa_{n+1} )
\end{eqnarray*}
is a well defined SIE subladder of (\ref{ladder}).
\end{lemma}

\begin{proof}
The identity (\ref{commutator condition}) ensures that the two different expressions in the definition of $\cE_n$ are equivalent.
Now let $v \in \cE_n$. Then
\begin{eqnarray*}
&& (A_{n+1}^+ A_{n+1}^- ) A_{n+1}^+  v \eq  A_{n+1}^+ ( A_{n+1}^- A_{n+1}^+ ) v \\ 
&=& A_{n+1}^+ \aa_{n+1} v \eq \aa_{n+1} A^+_{n+1} v,
\end{eqnarray*}
so $A_{n+1}^+  v \in \cE_{n+1}$.
On the other side we have
\begin{eqnarray*}
&& (A_{n}^- A_{n}^+ ) A_n^- v \eq A_n^- ( A_n^+ A_n^- ) v  \eq  A_n^- \aa_n v \eq \aa_n A_n^- v,
\end{eqnarray*}
but this is $A_n^- v \in \cE_{n-1}$.
\end{proof}

\section{The Weyl algebra}
The Weyl algebra is generated as an associative unital $\C$ algebra by the two operators $D,X$ subject to the canonical commutator relation
\[ [D,X] = 1.
\]
The standard representation is by operators acting on $\cC^\8(\R)$ via differentiation (D) and multiplication with the variable (X), respectively.
If $p$ is a polynomial over $\C$ , then
\[ [D,p(X)] = p'(X), \q\q [p(D),X] = p'(D).
\]
In the Weyl algebra we additionally define the operators
\begin{eqnarray*}
C &\defeq& D-1 \\
\cD_n &\defeq& 1 + D + D^2 + \dots + D^n
\end{eqnarray*}
Then
\begin{eqnarray*}
&& [C,D] = 0, \q\q [C,X] = 1, \q\q  C \cD_n = \cD_n C \eq D^{n+1} - 1.
\end{eqnarray*}

\section{The Laguerre grid}

In this section we assume that the operator $X$ in the Weyl algebra is invertible. On the algebraic level this can be achieved by adjoining
$X^{-1}$ to the Weyl algebra. With respect to the standard representation we have to restrict the domain of the relevant functions, to
$\R^+$, e.g.
\mip
Now, for $n \in \N_0, k \in\Z$ define the (second order differential) \emph{Laguerre operator}
\begin{eqnarray*}
H_{nk} &\defeq& H_{nk}^\7{Lag} \defeq XCD + (k+1)D + n \eq CXD + k D + n
% \\R_{nk}^\7{ver} &\defeq& D^{n+1}
\end{eqnarray*}
and --- according to some fixed representation of the Weyl algebra --- the spaces
\begin{eqnarray*}
U_{nk} \defeq \ker H_{nk} \,\cap\, \ker D^{n+1}.
\end{eqnarray*}

The diagram below shows locally the so-called \emph{Laguerre grid}. The global Laguerre grid is defined for $n \in \N_0, k \in\Z$.
Note that the second index $k \defeq m-n$ encodes the numbers of the NE/SW diagonals instead of the horizontal rows.

\def\plotSA #1 #2 #3 #4 /{ \arrow <1mm> [.25,.75] from #1 #2 to #3 #4 }
\begin{equation}
\label{Laguerre spaces grid local}
\def\SC #1 #2 {\setcoordinatesystem units <.5mm,.5mm> point at {-#1} {-#2} }
\beginpicture
\SC  00  00 \setplotarea x from -65 to +70, y from -70 to +60
\SC -60 -60 \put { $U_{n-1,k  }$ } [] at 00 00
\SC -60  00 \put { $U_{n-1,k+1}$ } [] at 00 00
\SC -60  60 \put { $U_{n-1,k+2}$ } [] at 00 00
\SC  00 -60 \put { $U_{n  ,k-1}$ } [] at 00 00
\SC  00  00 \put { $U_{n  ,k  }$ } [] at 00 00
\SC  00  60 \put { $U_{n  ,k+1}$ } [] at 00 00
\SC +60 -60 \put { $U_{n+1,k-2}$ } [] at 00 00
\SC  60  00 \put { $U_{n+1,k-1}$ } [] at 00 00
\SC +60 +60 \put { $U_{n+1,k  }$ } [] at 00 00

% East
\SC  00  00 \plotSA +20   +01   +40   +01   / \put {$\stD k   + XC $} [b]  at +30  03 %\plotEast{k  }{n+1}
\SC  00  60 \plotSA +20   +01   +40   +01   / \put {$\stD k+1 + XC $} [b]  at +30  03 %\plotEast{k+1}{n+1}
\SC -60  60 \plotSA +20   +01   +40   +01   / \put {$\stD k+2 + XC $} [b]  at +30  03 %\plotEast{k+2}{n  }
\SC -60  00 \plotSA +20   +01   +40   +01   / \put {$\stD k+1 + XC $} [b]  at +30  03 %\plotEast{k+1}{n  }
\SC -60 -60 \plotSA +20   +01   +40   +01   / \put {$\stD k   + XC $} [b]  at +30  03 %\plotEast{k  }{n  }
\SC  00 -60 \plotSA +20   +01   +40   +01   / \put {$\stD k-1 + XC $} [b]  at +30  03 %\plotEast{k-1}{n+1}

% West
\SC  00  00 \plotSA -20   -01   -40   -01   / \put {$\stD -D       $} [t]  at -30 -02 % \plotWest
\SC  60  00 \plotSA -20   -01   -40   -01   / \put {$\stD -D       $} [t]  at -30 -02 % \plotWest
\SC  60  60 \plotSA -20   -01   -40   -01   / \put {$\stD -D       $} [t]  at -30 -02 % \plotWest
\SC  00  60 \plotSA -20   -01   -40   -01   / \put {$\stD -D       $} [t]  at -30 -02 % \plotWest
\SC  00 -60 \plotSA -20   -01   -40   -01   / \put {$\stD -D       $} [t]  at -30 -02 % \plotWest
\SC  60 -60 \plotSA -20   -01   -40   -01   / \put {$\stD -D       $} [t]  at -30 -02 % \plotWest

% North
\SC  00  00 \plotSA -01   +20   -01   +40   / \put {$\stD -C       $} [Br]  at -02 +30 % \plotNort
\SC  60  00 \plotSA -01   +20   -01   +40   / \put {$\stD -C       $} [Br]  at -02 +30 % \plotNort
\SC -60  00 \plotSA -01   +20   -01   +40   / \put {$\stD -C       $} [Br]  at -02 +30 % \plotNort
\SC -60 -60 \plotSA -01   +20   -01   +40   / \put {$\stD -C       $} [Br]  at -02 +30 % \plotNort
\SC  00 -60 \plotSA -01   +20   -01   +40   / \put {$\stD -C       $} [Br]  at -02 +30 % \plotNort
\SC  60 -60 \plotSA -01   +20   -01   +40   / \put {$\stD -C       $} [Br]  at -02 +30 % \plotNort

% South
\SC  00  00 \plotSA +01   -20   +01   -40   / \put {$\stD \cD_{n  } $} [Bl]  at +02 -30 % \plotSout{n  }
\SC  60  00 \plotSA +01   -20   +01   -40   / \put {$\stD \cD_{n+1} $} [Bl]  at +02 -30 % \plotSout{n+1}
\SC  60  60 \plotSA +01   -20   +01   -40   / \put {$\stD \cD_{n+1} $} [Bl]  at +02 -30 % \plotSout{n+1}
\SC  00  60 \plotSA +01   -20   +01   -40   / \put {$\stD \cD_{n  } $} [Bl]  at +02 -30 % \plotSout{n  }
\SC -60  60 \plotSA +01   -20   +01   -40   / \put {$\stD \cD_{n-1} $} [Bl]  at +02 -30 % \plotSout{n-1}
\SC -60  00 \plotSA +01   -20   +01   -40   / \put {$\stD \cD_{n-1} $} [Bl]  at +02 -30 % \plotSout{n-1}

% Northeast
\SC  00  00 \plotSA +19.5 +20.5 +39.5 +40.5 / \put {$\stD n+k+1 + XC $} [rt] at +35  40 % \plotNoEa{n+k+1}{n+1}
\SC -60 -60 \plotSA +19.5 +20.5 +39.5 +40.5 / \put {$\stD n+k   + XC $} [rt] at +35  40 % \plotNoEa{n+k  }{n  }

% Southwest
\SC  00  00 \plotSA -19.5 -20.5 -39.5 -40.5 / \put {$\stD n   - XD   $} [lb] <-1mm,-2mm> at -35 -40 % \plotSoWe{n  }{n+k  }
\SC  60  60 \plotSA -19.5 -20.5 -39.5 -40.5 / \put {$\stD n+1 - XD   $} [lb] <-1mm,-2mm> at -35 -40 % \plotSoWe{n+1}{n+k+1}

% Southeast
\SC  00  00 \plotSA +19.5 -20.5 +39.5 -40.5 / \put {$\stD (k-1)\cD_{\!n  }\!- X$} [rb] <2mm,-3mm> at +35 -40 % \plotSoEa{(k-1)}{n}{n+1}
\SC -60  60 \plotSA +19.5 -20.5 +39.5 -40.5 / \put {$\stD (k+1)\cD_{\!n-1}\!- X$} [rb] <2mm,-3mm> at +35 -40 % \plotSoEa{(k+1)}{n-1}{n}

% Northwest
\SC  00  00 \plotSA -19.5 +20.5 -39.5 +40.5 / \put {$\stD DC         $} [lt] at -35  40 % \plotNoWe
\SC  60 -60 \plotSA -19.5 +20.5 -39.5 +40.5 / \put {$\stD DC         $} [lt] at -35  40 % \plotNoWe
\endpicture
\end{equation}

\newpage
\begin{theorem}[Laguerre grid]
The Laguerre grid is a well defined SIE grid.
The various circuit endomorphisms $U_{nk} \to U_{nk}$ act as scalars as follows
\begin{eqnarray*}
\bq\bq\bq 
\begin{array}[c]{ l @{} r @{\;\;=\;\;}  l }
\mbox{$\rightleftharpoons$ East}  & (-D) (k  +XC) & n + 1 \\[1.5ex]
\mbox{$\rightleftharpoons$ West}  & (-D) (k+1+XC) & n \\[1.5ex]
\mbox{$\rightleftharpoons$ North} & \cD_n (-C) & 1  \\[1.5ex]
\mbox{$\rightleftharpoons$ South} & (-C) \cD_n & 1 \\[1.5ex]
\mbox{$\rightleftharpoons$ Northeast} & (n+1 - XD ) (n+1+ k + XC)             & (n+1)(n+k+1) \\[1.5ex]
\mbox{$\rightleftharpoons$ Southwest} &  [ (n + k) + XC] [n - XD ]            & n(n+k)       \\[1.5ex]
\mbox{$\rightleftharpoons$ Northwest} & [(k+1)\cD_{n-1} - X ] DC              & n            \\[1.5ex]
\mbox{$\rightleftharpoons$ Southeast} & DC [(k-1)\cD_{n-1} - X ]              & n +1         \\[1.5ex]
\mbox{$\triangle$ East-Northeast}  &  (n+1 - XD ) (-C) ( k + XC)      & (n+1)(n+k+1) \\[1.5ex]
\mbox{$\triangle$ Northeast-East}  &  (-D) \cD_{n+1} (n+k+1+XC)       & n+1          \\[1.5ex]
\mbox{$\triangle$ Northeast-North} &  \cD_n (-D) (n+k+1 + XC )       & n+1          \\[1.5ex]
\mbox{$\triangle$ North-Northeast} &  (n+1-XD) (k+1+XC) (-C)         & (n+1)(n+k+1) \\[1.5ex]
\mbox{$\triangle$ North-Northwest} &  [(k+1) \cD_{n-1} - X ](-D)(-C) & n            \\[1.5ex]
\mbox{$\triangle$ Northwest-North} &  \cD_n (k+2+XC)  DC             & n            \\[1.5ex]
\mbox{$\triangle$ Northwest-West}  &  (k+1 + XC ) \cD_{n-1} D C       & n \\[1.5ex]
\mbox{$\triangle$ West-Northwest}  &  [(k+1)\cD_{n-1}-X](-C) (-D)     & n \\[1.5ex]
\mbox{$\Box \circlearrowleft $}    & \cD_n \, (-D) \, (-C) \, ( k + XC )   & n+1 \\[1.5ex]
\mbox{$\Box \circlearrowright$}    & (-D) \, \cD_{n+1} \, (k+1+XC) \, (-C) & n+1
\end{array}
\end{eqnarray*}
The double arrow symbolizes the narrow circuit endomorphisms in the direction as stated.
The triangle symbols $\triangle$ denote the various triangle circuit endomorphisms, the square symbols $\Box$ denote
the square circuit endomorphisms in clockwise and counterclockwise direction, respectively.
\end{theorem}

\begin{proof}
(1) The horizontal narrow circuit endomorphisms in East and West direction are:
\begin{eqnarray*}
A_{nk}^\Rq &=& (-D)(k + XC)     \eq - XCD - (k+1)D + 1 \eq -H_{nk} + n + 1 \\[1.5ex]
A_{nk}^\Lq &=& (k + 1 + XC)(-D) \eq - XCD - (k+1)D     \eq -H_{nk} + n
\end{eqnarray*}
When choosing the number sequence $\aa_{nk} = n$, we find that the ladder commutator condition (\ref{commutator condition}) in
Lemma \ref{Lemma: Scalar commutators} is fulfilled,
\begin{eqnarray}
\label{commutator condition Laguerre}
A_{nk}^\Rq - A_{nk}^\Lq &\eq& 1 \eq \aa_{n+1,k} - \aa_{nk}.
\end{eqnarray}
This Lemma shows that the (horizontal) SIE subladders
\begin{equation}
\label{Laguerre horizontal ladder I}
\beginpicture
\SC  00 00 \setplotarea x from -55 to +55, y from -03 to 03
\SC -50 00 \put {$\dots$}   [] at 00 00
\SC +50 00 \put {$\dots$}   [] at 00 00
\SC -30 00 \put {$\ker H_{n-1,k+1}$} [] <0mm,0mm> at -05 00
\SC  00 00 \put {$\ker H_{n   k  }$} [] at 00 00
\SC +30 00 \put {$\ker H_{n+1,k-1}$} [] <0mm,0mm> at +05 00
\SC -15 00 \arrow <1mm> [.25,.75] from -05 01 to +05  01 \put {\tiny $k+1+XC$} [b] at 00  02
           \arrow <1mm> [.25,.75] from +05 00 to -05  00 \put {\tiny $-D$} [t] at 00 -01
\SC +15 00 \arrow <1mm> [.25,.75] from -05 01 to +05  01 \put {\tiny $k + XC$} [b] at 00  02
           \arrow <1mm> [.25,.75] from +05 00 to -05  00 \put {\tiny $-D$} [t] at 00 -01
\endpicture
\end{equation}
with $\ker H_{nk} = \eig(A_{nk}^\Lq,\aa_{nk})$ are well defined. Because of
\begin{eqnarray*}
\bq\bq D^{n+2} (k + XC) &=& k D^{n+2} + (XD^{n+2} +(n+2)D^{n+1}) C \\ 
&=& [ XCD + k D + (n+2)C ] D^{n+1}  D^{n} (-D) \eq - D^{n+1} 
\end{eqnarray*}
we see that the even more restricted horizontal SIE subladders
\begin{equation}
\label{Laguerre horizontal ladder II}
\beginpicture
\SC  00 00 \setplotarea x from -55 to +55, y from -03 to 03
\SC -50 00 \put {$\dots$}   [] at 00 00
\SC +50 00 \put {$\dots$}   [] at 00 00
\SC -30 00 \put {$U_{n-1,k+1}$} [] <0mm,0mm> at -05 00
\SC  00 00 \put {$U_{n   k  }$} [] at 00 00
\SC +30 00 \put {$U_{n+1,k-1}$} [] <0mm,0mm> at +05 00
\SC -15 00 \arrow <1mm> [.25,.75] from -05 01 to +05  01 \put {\tiny $k+1+XC$} [b] at 00  02
           \arrow <1mm> [.25,.75] from +05 00 to -05  00 \put {\tiny $-D$} [t] at 00 -01
\SC +15 00 \arrow <1mm> [.25,.75] from -05 01 to +05  01 \put {\tiny $k + XC$} [b] at 00  02
           \arrow <1mm> [.25,.75] from +05 00 to -05  00 \put {\tiny $-D$} [t] at 00 -01
\endpicture
\end{equation}
are well defined.
\mip
(2) Because of
\begin{eqnarray*}
& &    H_{n,k+1} (-C) - (-C) H_{nk} \\ 
&=& [CXD + (k+1)D + n ] (-C) -\; (-C ) [XCD + (k+1)D + n ] \eq 0 \\[1.5ex]
& & H_{nk} \cD_n - \cD_n H_{n,k+1} \\ 
&=& [XCD + (k+1)D + n ] \cD_n  - \cD_n [CXD + (k+1)D + n ] \\
&=& XD (D^{n+1} - 1) - (D^{n+1} - 1) XD \\
&=& ( X D^{n+1} - D^{n+1} X ) D  \\
&=&  - (n+1) D^n D \eq 0
\end{eqnarray*}
the restricted vertical ladders
\begin{equation}
\label{Laguerre vertical ladder I}
\beginpicture
\SC  00 00 \setplotarea x from -55 to +55, y from -03 to 03
\SC -50 00 \put {$\dots$}   [] at 00 00
\SC +50 00 \put {$\dots$}   [] at 00 00
\SC -30 00 \put {$\ker H_{n-1,k}$} [] <0mm,0mm> at -05 00
\SC  00 00 \put {$\ker H_{n   k}$} [] at 00 00
\SC +30 00 \put {$\ker H_{n+1,k}$} [] <0mm,0mm> at +05 00
\SC -15 00 \arrow <1mm> [.25,.75] from -05 01 to +05  01 \put {\tiny $-C$} [b] at 00  02
           \arrow <1mm> [.25,.75] from +05 00 to -05  00 \put {\tiny $\cD_n$} [t] at 00 -01
\SC +15 00 \arrow <1mm> [.25,.75] from -05 01 to +05  01 \put {\tiny $-C$} [b] at 00  02
           \arrow <1mm> [.25,.75] from +05 00 to -05  00 \put {\tiny $\cD_{n+1}$} [t] at 00 -01
\endpicture
\end{equation}
and then, also the subladders
\begin{equation}
\label{Laguerre vertical ladder II}
\beginpicture
\SC  00 00 \setplotarea x from -55 to +55, y from -03 to 03
\SC -50 00 \put {$\dots$}   [] at 00 00
\SC +50 00 \put {$\dots$}   [] at 00 00
\SC -30 00 \put {$U_{n-1,k}$} [] <0mm,0mm> at -05 00
\SC  00 00 \put {$U_{n   k}$} [] at 00 00
\SC +30 00 \put {$U_{n+1,k}$} [] <0mm,0mm> at +05 00
\SC -15 00 \arrow <1mm> [.25,.75] from -05 01 to +05  01 \put {\tiny $-C$} [b] at 00  02
           \arrow <1mm> [.25,.75] from +05 00 to -05  00 \put {\tiny $\cD_n$} [t] at 00 -01
\SC +15 00 \arrow <1mm> [.25,.75] from -05 01 to +05  01 \put {\tiny $-C$} [b] at 00  02
           \arrow <1mm> [.25,.75] from +05 00 to -05  00 \put {\tiny $\cD_{n+1}$} [t] at 00 -01
\endpicture
\end{equation}
are well defined. The vertical circuit endomorphisms on $U_{nk}$ are
\begin{eqnarray*}
\cD_n (-C) &=& - D^{n+1} + 1 \eq 1 \\[1.5ex]
(-C) \cD_n &=& - D^{n+1} + 1 \eq 1
\end{eqnarray*}
This shows that (\ref{Laguerre vertical ladder II}) is an SIE subladder.

\mip
(3) The Northeast arrow in (\ref{Laguerre spaces grid local}) is well defined. Starting at $U_{nk}$, then going East and North yields the operator
\begin{eqnarray*}
     (-C) (k + XC)
&=& k + XC - (kD + DXC) \\
&=& k + XC + n - [kD + (XD + 1) C + n ] \\
&=& n + k + 1 + XC - [XCD + (k+1) D + n ] \\
&=& n + k + 1 + XC
\end{eqnarray*}

The Northeast narrow circuit endomorphism on $U_{nk}$ is
\begin{eqnarray*}
& & (n+1 - XD) ( n+1+k + XC) \\
&=& (n+1) ( n+1+k) + (n+1) XC - ( n+1+k ) XD - X D X C \\
&=& (n+1) ( n+1+k) - X [ (n+1) + k D + (XD + 1) C ] \\
&=& (n+1) ( n+1+k) - X H_{nk} \eq (n+1) ( n+1+k)
\end{eqnarray*}
\mip
(4) The Northeast-East triangle circuit endomorphism on $U_{nk}$ is
\begin{eqnarray*}
& & (-D) \cD_{n+1} (n+k+1 + XC ) \\ 
&=& - \cD_{n+1} [ (n+k+1)D + (XD+1) C ] \\
&=& - \cD_{n+1} [ XDC + (k+1)D + n + n C + C ] \\
&=& - \cD_{n+1} R_{nk} - (n+1) \cD_{n+1} C \\
&=& - (n+1) ( D^{n+2} - 1) \\
&=& n+1
\end{eqnarray*}
(5) The Northeast Square circuit endomorphism on $U_{nk}$ is
\begin{eqnarray*}
& & \cD_n \, (-D) \, (-C) \, ( k + XC ) \\
&=& \cD_n \, C  [ k D  + ( X D + 1) C ] \\
&=& ( D^{n+1} - 1) [ X C D + (k + 1 ) D + n - ( n + 1 ) ] \\
&=& - (n+1 ) ( D^{n+1} - 1) \\
&=& n+1
\end{eqnarray*}
All the other statements can be shown in a similar fashion, alternatively one can refer to Lemma \ref{Wide circuit endomorphisms}.
\end{proof}

\section{Associated Laguerre polynomials}
We have
\[  U_{0,0} \eq \ker R_{0,0} \;\cap\; \ker D \eq \ker ( XCD + D ) \;\cap\; \ker D  \eq \ker D.
\]
Within the standard representation $V = \cC^\8(\R)$ of the Weyl algebra we have $\ker D = \lgl 1\rgl$.
Thus, for $n,k\ge 0$ the spaces $U_{nk} = \lgl L_n^k \rgl$ are spanned by the associated Laguerre polynomials
\begin{eqnarray*}
L_n^k (x) &=& \sum_{j=0}^n {n+k \choose n-j} \frac{(-1)^j}{j!} x^j \eq
\f{(-1)^n}{n!} \, x^n + \dots + {n+k \choose n}.
\end{eqnarray*}

The diagram below shows the (monic) associated Laguerre polynomials in the grid for $n = 0,\dots 5$ and $m = n+k = 0,\dots,7$.
Note that all operators are only displayed up to scalar multiples.
For clarity we did not draw the diagonal arrows.

\def\doarE#1{
\arrow <1mm> [.25,.75] from 10  01 to 20 01
\arrow <1mm> [.25,.75] from 20 -01 to 10 -01
\put {\tiny $#1 + XC$} [b] at 15  02
\put {\tiny $-D$}      [t] at 15 -02
}
\def\doarN#1{
\arrow <1mm> [.25,.75] from -01 10 to -01 20
\arrow <1mm> [.25,.75] from  01 20 to  01 10
\put {\tiny $-C$}       [r] at -02 15
\put {\tiny $\cD_{#1}$} [l] at +02 15
}
\def\doarX#1{ }
\begin{eqnarray}
\bq\bq\bq 
\label{Laguerre spaces grid global}
\def\SC #1 #2 {\setcoordinatesystem units <.9mm,.7mm> point at {-#1} {-#2} }
\bq\bq\bq\bq\bq
\beginpicture
\def\ya#1{ \put {$\stC \langle #1 \rangle$} [] <0mm,0mm> at  00  00 }
\def\yb#1{ \put {$\stD \langle #1        $} [] <0mm,0mm> at -01  02 }
\def\yc#1{ \put {$\stD         #1 \rangle$} [] <0mm,0mm> at +01 -02 }
\def\yp#1{ \put {$\stD \langle #1        $} [] <0mm,0mm> at -01  04 }
\def\yq#1{ \put {$\stD         #1        $} [] <0mm,0mm> at -01  00 }
\def\yr#1{ \put {$\stD         #1 \rangle$} [] <0mm,0mm> at +01 -04 }

\SC 00 00 \setplotarea x from 00 to 160, y from 00 to 210
\SC  00  00 \doarE{ 0} \doarN{0} \ya{1}
\SC  30  00 \doarE{-1} \doarN{1} \ya{x}
\SC  60  00 \doarE{-2} \doarN{2} \ya{x^2}
\SC  90  00 \doarE{-3} \doarN{3} \ya{x^3}
\SC 120  00 \doarE{-4} \doarN{4} \ya{x^4}
\SC 150  00 \doarX{-5} \doarN{5} \ya{x^5}

\SC  00  30 \doarE{ 1} \doarN{0} \ya{1}
\SC  30  30 \doarE{ 0} \doarN{1} \ya{x-1}
\SC  60  30 \doarE{-1} \doarN{2} \ya{x^2-2x}
\SC  90  30 \doarE{-2} \doarN{3} \ya{x^3-3x^2}
\SC 120  30 \doarE{-3} \doarN{4} \ya{x^4-4x^3}
\SC 150  30 \doarX{-4} \doarN{5} \ya{x^5-5x^4}

\SC  00  60 \doarE{ 2} \doarN{0} \ya{1}
\SC  30  60 \doarE{ 1} \doarN{1} \ya{x-2}
\SC  60  60 \doarE{ 0} \doarN{2} \ya{x^2-4x+2}
\SC  90  60 \doarE{-1} \doarN{3} \yb{x^3-6x^2}\yc{\q +6x}
\SC 120  60 \doarE{-2} \doarN{4} \yb{x^4-8x^3} \yc{+12x^2}
\SC 150  60 \doarX{-3} \doarN{5} \yb{x^5-10x^4}\yc{+20x^3}

\SC  00  90 \doarE{ 3} \doarN{0} \ya{1}
\SC  30  90 \doarE{ 2} \doarN{1} \ya{x-3}
\SC  60  90 \doarE{ 1} \doarN{2} \ya{x^2-6x+6}
\SC  90  90 \doarE{ 0} \doarN{3} \yb{x^3-9x^2}   \yc{+18x-6}
\SC 120  90 \doarE{-1} \doarN{4} \yb{x^4-12x^3}  \yc{+36x^2-24x}
\SC 150  90 \doarX{-2} \doarN{5} \yb{x^5-15x^4}  \yc{+60x^3-60x^2}

\SC  00 120 \doarE{ 4} \doarN{0} \ya{1}
\SC  30 120 \doarE{ 3} \doarN{1} \ya{x-4}
\SC  60 120 \doarE{ 2} \doarN{2} \ya{x^2-8x+12}
\SC  90 120 \doarE{ 1} \doarN{3} \yb{x^3-12x^2} \yc{+36x-24}
\SC 120 120 \doarE{ 0} \doarN{4} \yb{x^4-16x^3+} \yc{72x^2-96x+24}
\SC 150 120 \doarX{-1} \doarN{5} \yp{x^5-20x^4} \yq{+120x^3} \yr{-240x^2+120x}

\SC  00 150 \doarE{ 5} \doarN{0} \ya{1}
\SC  30 150 \doarE{ 4} \doarN{1} \ya{x-5}
\SC  60 150 \doarE{ 3} \doarN{2} \ya{x^2-10x+20}
\SC  90 150 \doarE{ 2} \doarN{3} \yb{x^3-15x^2} \yc{+60x-60  }
\SC 120 150 \doarE{ 1} \doarN{4} \yp{x^4-20x^3} \yq{+120x^2}         \yr{-240x+120}
\SC 150 150 \doarX{ 0} \doarN{5} \yp{x^5-25x^4 \q} \yq{\q +200x^3 -600x^2} \yr{\q +600x-120}

\SC  00 180 \doarE{ 6} \doarN{0} \ya{1}
\SC  30 180 \doarE{ 5} \doarN{1} \ya{x-6}
\SC  60 180 \doarE{ 4} \doarN{2} \ya{x^2-12x+30}
\SC  90 180 \doarE{ 3} \doarN{3} \yb{x^3-18x^2} \yc{+90x-120}
\SC 120 180 \doarE{ 2} \doarN{4} \yp{x^4-24x^3} \yq{+180x^2} \yr{-480x+360}
\SC 150 180 \doarX{ 1} \doarN{5} \yp{x^5-30x^4 \q} \yq{\q +300x^3-1200x^2} \yr{\q +1800x-720}

\SC  00 210 \doarE{ 7} \doarX{0} \ya{1}
\SC  30 210 \doarE{ 6} \doarX{1} \ya{x-7}
\SC  60 210 \doarE{ 5} \doarX{2} \ya{x^2-14x+42}
\SC  90 210 \doarE{ 4} \doarX{3} \yb{x^3-21x^2} \yc{+126x-210}
\SC 120 210 \doarE{ 3} \doarX{4} \yp{x^4-28x^3} \yq{+252x^2} \yr{-840x+840}
\SC 150 210 \doarX{ 2} \doarX{5} \yp{x^5-35x^4 \q } \yq{\q +420x^3-2100x^2} \yr{\q +4200x+2520}

\SC  00  00 \put { $\stC n  = 0 $ } [t] at 00 -06
\SC  30  00 \put { $\stC n  = 1 $ } [t] at 00 -06
\SC  60  00 \put { $\stC n  = 2 $ } [t] at 00 -06
\SC  90  00 \put { $\stC n  = 3 $ } [t] at 00 -06
\SC 120  00 \put { $\stC n  = 4 $ } [t] at 00 -06
\SC 150  00 \put { $\stC n  = 5 $ } [t] at 00 -06

\SC  00  00  \put { $\stC k  = 0 \nearrow$   } [rt] at -05 -05
\SC  00  30  \put { $\stC k  = 1 \nearrow$   } [rt] at -05 -05
\SC  00  60  \put { $\stC k  = 2 \nearrow$   } [rt] at -05 -05
\SC  00  90  \put { $\stC k  = 3 \nearrow$   } [rt] at -05 -05
\SC  00 120  \put { $\stC k  = 4 \nearrow$   } [rt] at -05 -05
\SC  00 150  \put { $\stC k  = 5 \nearrow$   } [rt] at -05 -05
\SC  00 180  \put { $\stC k  = 6 \nearrow$   } [rt] at -05 -05
\SC  00 210  \put { $\stC k  = 7 \nearrow$   } [rt] at -05 -05
\endpicture
\end{eqnarray}

\newpage
When taking into account the norming factors for the Laguerre polynomials the local diagram (\ref{Laguerre spaces grid local})
has to be modified

\def\plotSA #1 #2 #3 #4 /{ \arrow <1mm> [.25,.75] from #1 #2 to #3 #4 }
\begin{equation}
\label{Laguerre polynomials grid local}
\def\SC #1 #2 {\setcoordinatesystem units <.5mm,.5mm> point at {-#1} {-#2} }
\beginpicture
\SC  00  00 \setplotarea x from -65 to +70, y from -70 to +60
\SC -60 -60 \put { $L_{n-1}^{k  }$ } [] at 00 00
\SC -60  00 \put { $L_{n-1}^{k+1}$ } [] at 00 00
\SC -60  60 \put { $L_{n-1}^{k+2}$ } [] at 00 00
\SC  00 -60 \put { $L_{n  }^{k-1}$ } [] at 00 00
\SC  00  00 \put { $L_{n  }^{k  }$ } [] at 00 00
\SC  00  60 \put { $L_{n  }^{k+1}$ } [] at 00 00
\SC +60 -60 \put { $L_{n+1}^{k-2}$ } [] at 00 00
\SC  60  00 \put { $L_{n+1}^{k-1}$ } [] at 00 00
\SC +60 +60 \put { $L_{n+1}^{k  }$ } [] at 00 00

% East
\SC  00  00 \plotSA +20   +01   +40   +01   / \put {$\stD \f{k   + XC}{n+1} $} [b]  at +30  03 %\plotEast{k  }{n+1}
\SC  00  60 \plotSA +20   +01   +40   +01   / \put {$\stD \f{k+1 + XC}{n+1} $} [b]  at +30  03 %\plotEast{k+1}{n+1}
\SC -60  60 \plotSA +20   +01   +40   +01   / \put {$\stD \f{k+2 + XC}{n  } $} [b]  at +30  03 %\plotEast{k+2}{n  }
\SC -60  00 \plotSA +20   +01   +40   +01   / \put {$\stD \f{k+1 + XC}{n  } $} [b]  at +30  03 %\plotEast{k+1}{n  }
\SC -60 -60 \plotSA +20   +01   +40   +01   / \put {$\stD \f{k   + XC}{n  } $} [b]  at +30  03 %\plotEast{k  }{n  }
\SC  00 -60 \plotSA +20   +01   +40   +01   / \put {$\stD \f{k-1 + XC}{n+1} $} [b]  at +30  03 %\plotEast{k-1}{n+1}

% West
\SC  00  00 \plotSA -20   -01   -40   -01   / \put {$\stD -D       $} [t]  at -30 -02 % \plotWest
\SC  60  00 \plotSA -20   -01   -40   -01   / \put {$\stD -D       $} [t]  at -30 -02 % \plotWest
\SC  60  60 \plotSA -20   -01   -40   -01   / \put {$\stD -D       $} [t]  at -30 -02 % \plotWest
\SC  00  60 \plotSA -20   -01   -40   -01   / \put {$\stD -D       $} [t]  at -30 -02 % \plotWest
\SC  00 -60 \plotSA -20   -01   -40   -01   / \put {$\stD -D       $} [t]  at -30 -02 % \plotWest
\SC  60 -60 \plotSA -20   -01   -40   -01   / \put {$\stD -D       $} [t]  at -30 -02 % \plotWest

% North
\SC  00  00 \plotSA -01   +20   -01   +40   / \put {$\stD      -C  $} [r]  at -02 +30 % \plotNort
\SC  60  00 \plotSA -01   +20   -01   +40   / \put {$\stD      -C  $} [r]  at -02 +30 % \plotNort
\SC -60  00 \plotSA -01   +20   -01   +40   / \put {$\stD      -C  $} [r]  at -02 +30 % \plotNort
\SC -60 -60 \plotSA -01   +20   -01   +40   / \put {$\stD      -C  $} [r]  at -02 +30 % \plotNort
\SC  00 -60 \plotSA -01   +20   -01   +40   / \put {$\stD      -C  $} [r]  at -02 +30 % \plotNort
\SC  60 -60 \plotSA -01   +20   -01   +40   / \put {$\stD      -C  $} [r]  at -02 +30 % \plotNort

% South
\SC  00  00 \plotSA +01   -20   +01   -40   / \put {$\stD \cD_{n  } $} [l]  at +02 -30 % \plotSout{n  }
\SC  60  00 \plotSA +01   -20   +01   -40   / \put {$\stD \cD_{n+1} $} [l]  at +02 -30 % \plotSout{n+1}
\SC  60  60 \plotSA +01   -20   +01   -40   / \put {$\stD \cD_{n+1} $} [l]  at +02 -30 % \plotSout{n+1}
\SC  00  60 \plotSA +01   -20   +01   -40   / \put {$\stD \cD_{n  } $} [l]  at +02 -30 % \plotSout{n  }
\SC -60  60 \plotSA +01   -20   +01   -40   / \put {$\stD \cD_{n-1} $} [l]  at +02 -30 % \plotSout{n-1}
\SC -60  00 \plotSA +01   -20   +01   -40   / \put {$\stD \cD_{n-1} $} [l]  at +02 -30 % \plotSout{n-1}

% Northeast
\SC  00  00 \plotSA +19.5 +20.5 +39.5 +40.5 / \put {$\stD \f{n+k+1 + XC}{n+1}  $} [rt] at +35  40 % \plotNoEa{n+k+1}{n+1}
\SC -60 -60 \plotSA +19.5 +20.5 +39.5 +40.5 / \put {$\stD \f{n+k   + XC}{n  }  $} [rt] at +35  40 % \plotNoEa{n+k  }{n  }

% Southwest
\SC  00  00 \plotSA -19.5 -20.5 -39.5 -40.5 / \put {$\stD \f{n   - XD}{n+k  }  $} [lb] <-1mm,-2mm> at -35 -40 % \plotSoWe{n  }{n+k  }
\SC  60  60 \plotSA -19.5 -20.5 -39.5 -40.5 / \put {$\stD \f{n+1 - XD}{n+k+1}  $} [lb] <-1mm,-2mm> at -35 -40 % \plotSoWe{n+1}{n+k+1}

% Southeast
\SC  00  00 \plotSA +19.5 -20.5 +39.5 -40.5 / \put {$\stD \f{(k-1)\cD_{\!n  } - X}{n+1}$} [rb] <2mm,-3mm> at +35  -40 % \plotSoEa{(k-1)}{n  }{n+1}
\SC -60  60 \plotSA +19.5 -20.5 +39.5 -40.5 / \put {$\stD \f{(k+1)\cD_{\!n-1} - X}{n  }$} [rb] <2mm,-3mm> at +35  -40 % \plotSoEa{(k+1)}{n-1}{n  }

% Northwest
\SC  00  00 \plotSA -19.5 +20.5 -39.5 +40.5 / \put {$\stD DC       $} [lt] at -35  40 % \plotNoWe
\SC  60 -60 \plotSA -19.5 +20.5 -39.5 +40.5 / \put {$\stD DC       $} [lt] at -35  40 % \plotNoWe
\endpicture
\end{equation}

The reader might wonder why we did not consider this diagram right from the start.
But when including the norm factors the condition (\ref{commutator condition Laguerre}) on the scalar ladder commutator
is not longer fulfilled. This condition is foundational for the whole ``Laguerre system''.

\mip
We can now demonstrate that many of the classical identities for associated Laguerre polynomials can be derived from the
diagram (\ref{Laguerre polynomials grid local}).

\begin{theorem}[Associated Laguerre polynomials]
For the associated Laguerre polynomials the following relations hold
\begin{itemize}
\item[(i)] Three term recurrence relation.
\begin{eqnarray*}
(2n+1+k - X ) L_n^k &=& ( n + 1 ) L_{n+1}^k + ( n + k ) L_{n-1}^k.
\end{eqnarray*}
\item[(ii)] There are various so called three-point-rules
\begin{eqnarray*} 
\bq\bq\bq\bq\bq\bq 
\begin{array}[c]{ r @{\;\eq\;} l @{\;\eq\;} l }
    L_n^k
& % \1n (k+1+XC) L_{n-1}^{k+1} \eq
    \1n (n+k+1+XC) L_{n-1}^{k+1} - L_{n-1}^{k+1}
& L_{n}^{k+1} - L_{n-1}^{k+1} \\[1.5ex]
    n L_n^k
& (n+k + XC) L_{n-1}^k
& (n+k) L_{n-1}^k - X L_{n-1}^{k+1} \\[1.5ex]
    k L_n^k
& (k + XC) L_n^k - XC L_n^k
& (n+1) L_{n+1}^{k-1} + X L_n^{k+1} \\[1.5ex]
    (n - X) L_n^k
& X C L_n^k + (n - X D) L_n^k
& - X L_n^{k+1} + (n+k) L_{n-1}^k
\end{array}
\end{eqnarray*}
\item[(iii)] Sheffer sequence. The North and West arrows pointing to $L_{n-1}^{k+1}$ yield the identity
\begin{eqnarray*}
D L_n^k &=& - L_{n-1}^{k+1} \eq (D-1) L_{n-1}^k.
\end{eqnarray*}
\item[(iv)] Reflection across the main diagonal $k = 0$.
\begin{eqnarray*}
L_{n+k}^{-k} &=& \f{n!}{(n+k)!} \, (-X)^k L_n^k
\end{eqnarray*}
\mip
\item[(v)] Rodrigues formula.
\begin{eqnarray*}
L_n^k(x) &=& \1{n!} X^{-k} C^n x^{n+k}
\end{eqnarray*}
Within the standard representation we have
\begin{eqnarray*}
L_n^k(x) &=& \1{n!} x^{-k} e^x \6^n (e^{-x} x^{n+k})
\end{eqnarray*}
\end{itemize}
\end{theorem}

\begin{proof}
(i) See the diagram (\ref{Laguerre polynomials grid local}).
The two diagonal arrows in Northeast and Southwest direction starting at $L_n^k$ contain the two equations
\begin{eqnarray*}
( n + 1 ) L_{n+1}^k &=& (n+1+k + XD - X ) L_n^k \\
( n + k ) L_{n-1}^k &=& (n     - XD     ) L_n^k .
\end{eqnarray*}
Adding up the two equations yields the three term recurrence relation in (i).
\mip
(ii) and (iii) were already shown in the theorem.
\mip
(iv) For fixed $n \in\N_0$ we prove by induction over $k \in\N_0$
\begin{eqnarray*}
U_{nk} \q \sr{(-X)^k} \llongrightarrow \q U_{n+k,-k}.
\end{eqnarray*}
This operator jumps over $|k|$ steps in Southeast ($k \ge 0$) resp.\ Northwest ($k \le 0$) direction.
\mip
The initial case $k = 0$ is trivial.
In order to show the induction step $k-1 \mto k$  we consider the following chain of operators
\begin{eqnarray*}
U_{n k  }         \sr{\cD_n     } \llongrightarrow 
U_{n,k-1}         \sr{(-X)^{k-1}} \llongrightarrow 
U_{n+k-1,-(k-1)}  \sr{-(k-1)+XC } \llongrightarrow 
U_{n+k,-k}
\end{eqnarray*}
The operator on the middle arrow is well defined according to the induction hypothesis.
Then on $U_{nk}$ we compute the composition
\begin{eqnarray*}
\bq\bq  [-(k-1) + X C] \, (-X)^{k-1} \, \cD_n
&=& X (-1)^{k-1}[-(k-1) X^{k-2} + CX^{k-1}] \cD_n \\
&=& X (-1)^{k-1} X^{k-1} \, C \, \cD_n \\ 
&=& - (-X)^k (D^{n+1} - 1) \eq (-X)^k
\end{eqnarray*}
For $k \le 0 $ the reflection is realized by the inverse multiplication operator $(-X)^k$.

Inspection of the leading coefficient then shows that
\begin{eqnarray*}
L_{n+k}^{-k} &=& \f{n!}{(n+k)!} \, (-X)^k L_n^k.
\end{eqnarray*}

(v) In order to compute $L_n^k \in U_{nk}$ we start at $U_{n+k,-(n+k)}$ (bottom line in the above diagram (\ref{Laguerre spaces grid global})), go $k$ steps north, then reflect across the main diagonal, altogether
\begin{eqnarray*}
U_{n+k,-(n+k)} \sr{C^n}\llongrightarrow  U_{n+k,-k} \sr{X^{-k}}\llongrightarrow U_{nk}
\end{eqnarray*}
Since $U_{n+k,-(n+k)} = \lgl x^{n+k} \rgl $ we find that
\begin{eqnarray*}
L_n^k(x) &=& \1{n!} X^{-k} C^n x^{n+k}
\end{eqnarray*}
One can figure out the scalar factor by comparison of the leading coefficients.

\mip
Within the standard representation the operator $C = D - 1$ acts by $f(x) \mto e^x \6 [e^{-x} f(x)]$, hence
the operator $C^n$ acts by $f(x) \mto e^x \6^n [e^{-x} f(x)]$.
\end{proof}

\section{The Legendre-Gegenbauer grid}
% \subsection{Additional operator in the Weyl algebra}
In the Weyl algebra we additionally define the operator
\begin{eqnarray*}
R &\defeq& X^2 - 1
\end{eqnarray*}
It fulfills the commutator relation
\begin{eqnarray}
\label{commutator DR}
[D,R^j] &\eq& 2jX R^{j-1}, \q\q j  \in \N_0.
\end{eqnarray}

For $n,\ell \in\R$ define the (second order differential) \emph{Legendre operator}
\begin{eqnarray*}
%H_{n\ell}^\7{hor} &\defeq&   R [ RD^2 + (\ell+1)X D - n (n + \ell) ] \\
%H_{n\ell}^\7{ver} &\defeq& X^2 [ RD^2 + (\ell+1)X D - n (n + \ell) ] \\
H_{n\ell} &\defeq& H_{n\ell}^\7{Leg} \defeq RD^2 + (\ell+1)X D - n (n + \ell). % H_{n\ell}^\7{ver} - H_{n\ell}^\7{hor} \eq
\end{eqnarray*}
and its kernel
\begin{eqnarray*}
U_{n\ell} &\defeq& \ker H_{n\ell} \eq \ker X^2 H_{n\ell} \eq \ker R H_{n\ell}.
\end{eqnarray*}
Here we assume that the operators $X^2$ and $R$ act as injective operators on the given representation space.
This is fulfilled for the standard representation of the Weyl algebra on $\cC^\8(\R)$ or any suitable subspace.

\mip
The diagram below shows the local structure of the \emph{Legendre-Gegenbauer grid} in the Weyl algebra.

\begin{equation}
\label{Legendre spaces grid local}
\def\SC #1 #2 {\setcoordinatesystem units <.6mm,.6mm> point at {-#1} {-#2} }
\def\plotSA #1 #2 #3 #4 /{ \arrow <1mm> [.25,.75] from #1 #2 to #3 #4 }
\beginpicture
\SC -60 -60 \put { $U_{n-1,\ell-2}$ } [] at 00 00
\SC -60  00 \put { $U_{n-1,\ell  }$ } [] at 00 00
\SC -60  60 \put { $U_{n-1,\ell+2}$ } [] at 00 00
\SC  00 -60 \put { $U_{n  ,\ell-2}$ } [] at 00 00
\SC  00  00 \put { $U_{n  ,\ell  }$ } [] at 00 00
\SC  00  60 \put { $U_{n  ,\ell+2}$ } [] at 00 00
\SC +60 -60 \put { $U_{n+1,\ell-2}$ } [] at 00 00
\SC  60  00 \put { $U_{n+1,\ell  }$ } [] at 00 00
\SC +60 +60 \put { $U_{n+1,\ell+2}$ } [] at 00 00

% East
\SC  00  00 \plotSA +20 +01 +40 +01 / \put {$\stD (n+\ell  )X + RD $} [b] <0mm,0mm> at +30 +03
\SC  00  60 \plotSA +20 +01 +40 +01 / \put {$\stD (n+\ell+2)X + RD $} [b] <0mm,0mm> at +30 +03
\SC -60  60 \plotSA +20 +01 +40 +01 / \put {$\stD (n+\ell+1)X + RD $} [b] <0mm,0mm> at +30 +03
\SC -60  00 \plotSA +20 +01 +40 +01 / \put {$\stD (n+\ell-1)X + RD $} [b] <0mm,0mm> at +30 +03
\SC -60 -60 \plotSA +20 +01 +40 +01 / \put {$\stD (n+\ell-3)X + RD $} [b] <0mm,0mm> at +30 +03
\SC  00 -60 \plotSA +20 +01 +40 +01 / \put {$\stD (n+\ell-2)X + RD $} [b] <0mm,0mm> at +30 +03

% West
\SC  00  00 \plotSA -20 -01 -40 -01 / \put {$\stD  n    X-RD  $} [t] at -30 -03
\SC  60  00 \plotSA -20 -01 -40 -01 / \put {$\stD (n+1) X-RD  $} [t] at -30 -03
\SC  60  60 \plotSA -20 -01 -40 -01 / \put {$\stD (n+1) X-RD  $} [t] at -30 -03
\SC  00  60 \plotSA -20 -01 -40 -01 / \put {$\stD  n    X-RD  $} [t] at -30 -03
\SC  00 -60 \plotSA -20 -01 -40 -01 / \put {$\stD  n    X-RD  $} [t] at -30 -03
\SC  60 -60 \plotSA -20 -01 -40 -01 / \put {$\stD (n+1) X-RD  $} [t] at -30 -03

% North
\SC  00  00 \plotSA -01 +20 -01 +40 / \put {$\stD XD+n+\ell   $} [r] <0mm,+4mm> at -03 +30
\SC  60  00 \plotSA -01 +20 -01 +40 / \put {$\stD XD+n+\ell+1 $} [r] <0mm,+4mm> at -03 +30
\SC -60  00 \plotSA -01 +20 -01 +40 / \put {$\stD XD+n+\ell-1 $} [r] <0mm,+4mm> at -03 +30
\SC -60 -60 \plotSA -01 +20 -01 +40 / \put {$\stD XD+n+\ell-3 $} [r] <0mm,+4mm> at -03 +30
\SC  00 -60 \plotSA -01 +20 -01 +40 / \put {$\stD XD+n+\ell-2 $} [r] <0mm,+4mm> at -03 +30
\SC  60 -60 \plotSA -01 +20 -01 +40 / \put {$\stD XD+n+\ell-1 $} [r] <0mm,+4mm> at -03 +30

% South
\SC -60  00 \plotSA +01 -20 +01 -40 / \put {$\stD R(\!XD - n+1) + \ell-1 $} [l] <0mm,-4mm> at +03 -30
\SC  00  00 \plotSA +01 -20 +01 -40 / \put {$\stD R(\!XD - n  ) + \ell-1 $} [l] <0mm,-4mm> at +03 -30
\SC  60  00 \plotSA +01 -20 +01 -40 / \put {$\stD R(\!XD - n-1) + \ell-1 $} [l] <0mm,-4mm> at +03 -30
\SC -60  60 \plotSA +01 -20 +01 -40 / \put {$\stD R(\!XD - n+1) + \ell+1 $} [l] <0mm,-4mm> at +03 -30
\SC  00  60 \plotSA +01 -20 +01 -40 / \put {$\stD R(\!XD - n  ) + \ell+1 $} [l] <0mm,-4mm> at +03 -30
\SC  60  60 \plotSA +01 -20 +01 -40 / \put {$\stD R(\!XD - n-1) + \ell+1 $} [l] <0mm,-4mm> at +03 -30

% Northwest
\SC  00  00  \plotSA -19.5 +20.5 -39.5 +40.5 / \put {$\stD D $} [lt] at -35  40
\SC  60 -60  \plotSA -19.5 +20.5 -39.5 +40.5 / \put {$\stD D $} [lt] at -35  40

% Southeast
\SC  00 00 \plotSA +19.5 -20.5 +39.5 -40.5 / \put {$\stD (\ell-1)X+RD $} [rb] <3mm,-3mm> at +35 -40
\SC -60 60 \plotSA +19.5 -20.5 +39.5 -40.5 / \put {$\stD (\ell+1)X+RD $} [rb] <3mm,-3mm> at +35 -40

% Northeast
%\SC  00  00  \plotSA +19.5 +20.5 +39.5 +40.5 / \put {$\stD          $} [rt] at +35  40 % \plotNoEa{n+\ell}
%\SC -60 -60  \plotSA +19.5 +20.5 +39.5 +40.5 / \put {$\stD          $} [rt] at +35  40 % \plotNoEa{n+\ell-2}

% Southwest
%\SC  00  00 \plotSA -19.5 -20.5 -39.5 -40.5 / \put {$\stD          $} [lb] at -35 -40 % \plotSoWe{n}{\ell-1}
%\SC  60  60 \plotSA -19.5 -20.5 -39.5 -40.5 / \put {$\stD          $} [lb] at -35 -40 % \plotSoWe{n}{\ell+1}
\endpicture
\end{equation}

\newpage
\begin{theorem}[Legendre-Gegenbauer grid]
The Legendre-Gegenbauer grid is a well defined SIE grid.
The various circuit endomorphisms $U_{n\ell} \to U_{n\ell}$ act as scalars as follows
\begin{small}
\begin{eqnarray*}
\bq\bq
\begin{array}[c]{ l @{} r c l }
\mbox{$\rightleftharpoons$ East}      & [(n+1) X - R D ] \cd [ (n+\ell) X + R D ]    &=& (n+1)(n+\ell)          \\[1.5ex]
\mbox{$\rightleftharpoons$ West}      & [ (n+\ell-1) X + R D ] \cd [ n X - R D ]     &=&  n   (n+\ell-1)        \\[1.5ex]
\mbox{$\rightleftharpoons$ North}     & [R(XD-n)+ (\ell+1) ] \, (XD + n+\ell)        &=& (n+\ell  ) (n+\ell+1)  \\[1.5ex]
\mbox{$\rightleftharpoons$ South}     & (XD + n+\ell-2) [R(XD-n)+ (\ell-1) ]         &=& (n+\ell-2) (n+\ell-1)  \\[1.5ex]
\mbox{$\rightleftharpoons$ Northwest} & [ (\ell+1) X + R D ] \, D                    &=&   n  (n+\ell)          \\[1.5ex]
\mbox{$\rightleftharpoons$ Southeast} & D \, [ (\ell-1) X + R D ]                    &=& (n+1)(n+\ell-1)        \\[1.5ex]
\mbox{$\triangle$ N-NW}               & [(\ell+1)X + RD]  \, (nX-RD) \,(XD+n+\ell)   &=& n (n+\ell  )(n+\ell+1) \\[1.5ex]
\mbox{$\triangle$ NW-N}               & [R(XD-n)+\ell+1]  \, [(n+\ell+1)X+RD] \, D   &=& n (n+\ell  )(n+\ell+1) \\[1.5ex]
\mbox{$\triangle$ NW-W}               & [(n+\ell-1)X+RD ] \, (XD+n+\ell+1)    \, D   &=& n (n+\ell-1)(n+\ell  ) \\[1.5ex]
\mbox{$\triangle$ W-NW}               & [(\ell+1)X+RD ]   \, (XD+n+\ell-1) \,(nX-RD) &=& n (n+\ell-1)(n+\ell  ) \\[1.5ex]
\mbox{$\Box \circlearrowleft $}       & [ R(XD-n)+\ell+1 ] \, [(n+1)X-RD] \q\q                                  \\ {}
                                    & [XD+n+\ell+1] \, [(n+\ell)X + RD ]           &=& (n+1) (n+\ell)  \\
                                    &                                              & & \q (n+\ell+1)(n+\ell+2) \\[1.5ex]
\mbox{$\Box \circlearrowright$}       & [ (n+1)X-RD] \, [R(XD-(n+1))+ \ell+1]  \\ {}
                                    & [(n+\ell+2)X +RD] \, (XD+n+\ell)             &=& (n+1) (n+\ell)         \\
                                    &                                              & & \q (n+\ell+1) (n+\ell+2)
\end{array}
\end{eqnarray*}
\end{small}
\end{theorem}

\begin{proof}
(1) The horizontal narrow circuit endomorphisms in East and West direction are:
\begin{eqnarray*}
    A_{n\ell}^\Rq
&=& [ (n+1) X - R D ] \, [ (n+\ell) X + R D ] \\
&=& (n+1)(n+\ell) X^2 - n (R D X - X R D ) - R D R D \\ 
& & -\; \ell R D X + X R D \\
&=& (n+1)(n+\ell) (R+1) - n R - R (RD+2X) D \\ 
& & - \; \ell R ( XD +1) + X R D \\
&=& - R H_{n\ell} + (n+1)(n+\ell) \\[1.5ex]
% &=& (n+1)(n+\ell) \\
    A_{n\ell}^\Lq
& & [ (n+\ell-1) X + R D ] \, ( n X - R D ) \\
&=& n(n+\ell-1) X^2 + n R [D,X] - R D R D - (\ell-1) X R D \\
&=& n(n+\ell-1) (R +1) + n R - R (RD + 2X) D - (\ell-1) X R D \\
&=& - R H_{n\ell} + n (n+\ell-1)
% \\ &=&  n (n+\ell-1)
\end{eqnarray*}

When choosing the number sequence $\aa_{n\ell} = n (n+\ell-1)$, we find that the commutator condition (\ref{commutator condition}) in
Lemma \ref{Lemma: Scalar commutators} is fulfilled,
\begin{eqnarray}
\label{commutator condition Laguerre horizontal}
A_{n\ell}^\Rq - A_{n\ell}^\Lq &\eq& (n+1)(n+\ell) - n (n+\ell-1) \eq \aa_{n+1,\ell} - \aa_{n\ell}.
\end{eqnarray}
Then the (horizontal) SIE subladders
\begin{equation}
\label{Legendre horizontal ladder}
\beginpicture
\SC  00 00 \setplotarea x from -55 to +55, y from -03 to 03
\SC -50 00 \put {$\dots$}   [] at 00 00
\SC +50 00 \put {$\dots$}   [] at 00 00
\SC -30 00 \put {$U_{n-1,\ell}$} [] <0mm,0mm> at -05 00
\SC  00 00 \put {$U_{n   \ell}$} [] at 00 00
\SC +30 00 \put {$U_{n+1,\ell}$} [] <0mm,0mm> at +05 00
\SC -15 00 \arrow <1mm> [.25,.75] from -08 01 to +08  01 \put {\tiny $(n+\ell-1) X + R D $} [b] at 00  02
           \arrow <1mm> [.25,.75] from +08 00 to -08  00 \put {\tiny $ n X - R D         $} [t] at 00 -01
\SC +15 00 \arrow <1mm> [.25,.75] from -08 01 to +08  01 \put {\tiny $ (n+\ell) X + R D  $} [b] at 00  02
           \arrow <1mm> [.25,.75] from +08 00 to -08  00 \put {\tiny $ (n+1) X - R D     $} [t] at 00 -01
\endpicture
\end{equation}
are well defined.
\mip
(2) The vertical narrow circuit endomorphisms in North and South direction are:
\begin{eqnarray*}
    B_{n\ell}^\Rq
&=& [R(XD-n)+ (\ell+1) ](XD + n+\ell) \\
&=& R XDXD + (n+\ell) R XD - n R XD - n(n+\ell)R \\
& & +\; (\ell+1) XD + (\ell+1)(n+\ell) \\
&=& R XDXD + \ell (X^2-1)XD - n(n+\ell)(X^2-1) \\ 
& & + \;  (\ell+1) XD + (\ell+1)(n+\ell) \\
&=& R X( XD+1)D + \ell X^3D - n(n+\ell)X^2 \\ 
& & + \; XD + (n+\ell)(n+\ell+1) \\ 
&=& X^2 [R D^2 + (\ell+1)XD - n (n+\ell) - X D ] \\
& & + \; R X D + X D + (n+\ell)(n+\ell+1) \\
&=& X^2 H_{n\ell} + (n+\ell)(n+\ell+1) \\[1.5ex]
% &=& (n+\ell)(n+\ell+1) \\
    B_{n\ell}^\Lq
&=& (XD + n+\ell-2) [R(XD-n)+ (\ell-1) ] \\
&=& X[D(RX)]D - n XDR + (\ell-1) XD + (n+\ell-2)RXD \\
& &  \hs{1cm} - n (n+\ell-2)R + (n+\ell-2)(\ell-1) \\
&=&  X[(RX)D + (3X^2 - 1)] D - n X(RD +2X) \\ 
& & +\; (\ell-1) XD + (n+\ell-2)RXD \\
& & \hs{1cm} - n (n+\ell-2)(X^2-1) + (n+\ell-2)(\ell-1) \\
&=& X^2 R D^2 + 3 X^3 D + (\ell-2) XD + (\ell-2) R XD - 2nX^2 \\
&& \hs{1cm} - n (n+\ell-2)X^2 + (n+\ell-2)(\ell-1) \\
&=& X^2 H_{n\ell} + (n+\ell-1)(n+\ell-2)
% &=& \\(n+\ell-1)(n+\ell-2)
\end{eqnarray*}
Again the difference of the two circuit operators is a scalar. With $\ba_{n\ell} = (n+\ell-1)(n+\ell-2)$ we find that
the commutator condition (\ref{commutator condition}) in Lemma \ref{Lemma: Scalar commutators} is fulfilled,
\begin{eqnarray}
\label{commutator condition Laguerre vertical}
B_{n\ell}^\Rq - B_{n\ell}^\Lq 
&\eq& (n+\ell)(n+\ell+1) - (n+\ell-1)(n+\ell-2) \\
\nn &=& \ba_{n,\ell+2} - \ba_{n\ell}.
\end{eqnarray}
Lemma \ref{Lemma: Scalar commutators} implies that the vertical ladders in the Legendre-Gegenbauer grid
\begin{equation}
\def\SC #1 #2 { \setcoordinatesystem units <1mm,1mm> point at {-#1} {-#2} } 
\label{Legendre horizontal ladder}
\beginpicture
\SC  00 00 \setplotarea x from -55 to +55, y from -03 to 03
\SC -50 00 \put {$\dots$}   [] at 00 00
\SC +50 00 \put {$\dots$}   [] at 00 00
\SC -30 00 \put {$U_{n,\ell-2}$} [] <0mm,0mm> at -05 00
\SC  00 00 \put {$U_{n \ell  }$} [] at 00 00
\SC +30 00 \put {$U_{n,\ell+2}$} [] <0mm,0mm> at +05 00
\SC -15 00 \arrow <1mm> [.25,.75] from -08 01 to +08  01 \put {\tiny $XD + n  + \ell -2$} [b] at 00  02
           \arrow <1mm> [.25,.75] from +08 00 to -08  00 \put {\tiny $R(XD-n) + \ell -1$} [t] at 00 -01
\SC +15 00 \arrow <1mm> [.25,.75] from -08 01 to +08  01 \put {\tiny $XD + n  + \ell   $} [b] at 00  02
           \arrow <1mm> [.25,.75] from +08 00 to -08  00 \put {\tiny $R(XD-n) + \ell +1$} [t] at 00 -01
\endpicture
\end{equation}
are well defined.

\mip
(3) The diagonal narrow circuit endomorphisms in Southeast and Northwest direction are
\begin{eqnarray*}
    D [(\ell-1)X + RD ]
&=& (RD + 2X) D + (\ell-1) (XD +1) \\
&=& RD^2 + (\ell+1) XD + \ell - 1 \\
&=& H_{n\ell} + (n+1)(n+\ell-1) \\[1.5ex]
[(\ell+1)X + RD ] D
&=& H_{n\ell} + n (n+\ell)
\end{eqnarray*}
Once more the difference of the two circuit operators is a scalar. With $\eta_{n\ell} = n(n+\ell)$ we find that
the commutator condition (\ref{commutator condition}) in Lemma \ref{Lemma: Scalar commutators} is fulfilled,
\begin{eqnarray}
\label{commutator condition Laguerre diagonal}
\nn & & D [(\ell-1)X + RD ] - [(\ell+1)X + RD ] D \\
&=& (n+1)(n+\ell-1) - n(n+\ell)  \\ 
\nn &=& \ell - 1 \eq \eta_{n+1,\ell-2} - \eta_{n\ell}.
\end{eqnarray}
Lemma \ref{Lemma: Scalar commutators} implies that the diagonal ladders in the Legendre-Gegenbauer grid
\begin{equation}
\def\SC #1 #2 { \setcoordinatesystem units <1mm,1mm> point at {-#1} {-#2} } 
\label{Legendre diagonal ladder}
\beginpicture
\SC  00 00 \setplotarea x from -55 to +55, y from -03 to 03
\SC -52 00 \put {$\dots$}   [] at 00 00
\SC +52 00 \put {$\dots$}   [] at 00 00
\SC -30 00 \put {$U_{n-1,\ell+2}$} [] <0mm,0mm> at -05 00
\SC  00 00 \put {$U_{n   \ell  }$} [] at 00 00
\SC +30 00 \put {$U_{n+1,\ell-2}$} [] <0mm,0mm> at +05 00
\SC -15 00 \arrow <1mm> [.25,.75] from -08 01 to +08  01 \put {\tiny $ (\ell+1) X + RD $} [b] at 00  02
           \arrow <1mm> [.25,.75] from +08 00 to -08  00 \put {\tiny $   D             $} [t] at 00 -01
\SC +15 00 \arrow <1mm> [.25,.75] from -08 01 to +08  01 \put {\tiny $ (\ell-1) X + RD $} [b] at 00  02
           \arrow <1mm> [.25,.75] from +08 00 to -08  00 \put {\tiny $   D             $} [t] at 00 -01
\endpicture
\end{equation}
are well defined SIE subladders.
\mip
(4) Note that the Northeast square is commutative.
\begin{eqnarray*}
& & (XD + n+\ell+1)[(n+\ell)X + RD ] \\ 
& & \q - \; [(n+\ell+2)X + RD] (XD + n + \ell) \\ 
&=& (n+\ell) XDX - (n+\ell+2) XXD + XDRD \\ 
& & \q -\; RDXD + (n+\ell+1) RD \\
& & \q - \; (n+\ell) RD + [(n+\ell+1)(n+\ell) - (n+\ell+2)(n+\ell)] X   \\
&=& (n+\ell) X (DX - XD) - 2XD + X(RD+2X) D \\
& & \q\q - \; R(XD+1) D + RD - (n+\ell) X \;\eq\; 0
\end{eqnarray*}
Then, with steps (1) and (2) we get for the Northeast square endomorphism in counterclockwise direction
\begin{eqnarray*}
& & [ R(XD-n)+\ell+1 ] \, [(n+1)X-RD] \,\\ 
& & \q  [XD+n+\ell+1] \, [(n+\ell)X + RD ] \\
&=& [ R(XD-n)+\ell+1 ] \, [(n+1)X-RD] \\
& & \q [(n+\ell+2)X + RD] (XD + n + \ell) \\
&=& [ R(XD-n)+\ell+1 ] \, (n+1) (n+\ell+2) \, (XD + n + \ell) \\
&=& (n+1) (n+\ell ) (n+\ell+1) (n+\ell+2)
\end{eqnarray*}
and in clockwise direction
\begin{eqnarray*}
& & [ (n+1)X-RD] \, [R(XD-(n+1))+\ell+1] \\
& & \q [(n+\ell+2)X +RD] \, (XD + n + \ell) \\
&=& [ (n+1)X-RD] \, [R(XD-(n+1))+\ell+1] \\ 
& & (XD + n+\ell+1)   \, [(n+\ell)X + RD ] \\
&=& [ (n+1)X-RD] \, (n+\ell+1) (n+\ell+2)  \, [(n+\ell)X + RD ] \\
&=& (n+1) (n+\ell  ) (n+\ell+1) (n+\ell+2)
\end{eqnarray*}
(5) The Northwest square on $U_{n\ell}$ is commutative up to scalar factors.
\begin{eqnarray*}
& & (n+\ell-1) (nX - RD ) ( XD + n+\ell) \\ 
& & \q - \; (n+\ell+1) ( XD + n+\ell-1)(nX - RD ) \\
&=& (n+\ell-1) [nX^2 D + n(n+\ell) X - RDXD - (n+\ell)RD] \\
& & \q - \; (n+\ell+1)[n X D X + n(n+\ell-1) X - XDRD - (n+\ell-1)RD] \\
&=& n(n+\ell-1) X^2 D - n(n+\ell+1) XDX + [n(n+\ell)(n+\ell-1) \\
& & \q - \; n (n+\ell+1)(n+\ell-1)] X - (n+\ell-1) R(XD+1) D \\
& & \q + \; (n+\ell+1) X (RD+2X) D \\ 
& & \q + \; [-(n+\ell-1)(n+\ell) + (n+\ell+1)(n+\ell-1)]RD \\
&=& n(n+\ell+1)X ( XD-DX) - 2n X^2 D - n(n+\ell-1)X \\
& & \q + \; (n+\ell) (2X^2 D - RD) + R(XD+1) D \\ 
& & \q + \; X (RD + 2X) D + (n+\ell-1) RD \\
&=& -2n(n+\ell) X - 2nX^2 D + (n+\ell) 2X^2 D \\ 
& & \q - \;RD + RXD^2 + RD + XRD^2 + 2X^2 D \\
&=& 2X [RD^2 + (\ell+1)XD - n(n+\ell)] \eq 2 X H_{n\ell} \eq 0
\end{eqnarray*}

(6) The North Northwest triangle circuit endomorphism acts on $U_{n\ell}$ as follows
\begin{eqnarray*}
& & [(\ell+1)X + RD] \, (nX - RD) (XD + n + \ell) \\
&=& [(\ell+1)X + RD] \, [n X^2 D + n(n+\ell) X - R(XD+1) D - (n+\ell) RD] \\
&=& [(\ell+1)X + RD] \, \Big[ (-X) [RD^2 + (\ell+1) XD - n(n+\ell)] \\ 
& & \q + \; (n+\ell+1)X^2 D - (n+\ell+1)RD \Big] \\
&=& [(\ell+1)X + RD] \, [-X H_{n\ell} + (n+\ell+1) D ] \\
&=& (n+\ell+1) [RD^2 + (\ell+1) XD - n(n+\ell) ] \\ 
& & \q +\; n (n+\ell) (n+\ell+1) - [(\ell+1) X + RD ] X H_{n\ell} \\
&=& [n+\ell+1 - (\ell+1)X^2 - R D X ] H_{n\ell} + n(n+\ell) (n+\ell+1) \\
&=& n(n+\ell) (n+\ell+1)
\end{eqnarray*}

We skip all the other calculations, since they can be done in a similar manner or by reference to Lemma \ref{Wide circuit endomorphisms}.
\end{proof}

\section{The Legendre-Gegenbauer polynomials}
We define
\[  W_{0,0} \eq \ker H_{0,0} \;\cap\; \ker D \eq \ker ( XCD + D ) \;\cap\; \ker D  \eq \ker D.
\]
This space generates an SIE subladder $(W_{nk})$ of the subladder $(U_{nk})$.
Within the standard representation $V = \cC^\8(\R)$ of the Weyl algebra we have $\ker D = \lgl 1\rgl$.
Thus, for $n,k\ge 0$ the spaces $W_{nk}$ are one--dimensional.
For $n\ge 0$ the one--dimensional spaces $W_{n\ell}$ exactly contain the Legendre-Gegenbauer polynomials
\begin{eqnarray*}
P_n^\ell(x)
&=& \sum_{j=0}^{\lfloor\2n\rfloor} (-1)^j \f {\Ga(n+\2\ell-j)}{\Ga(\2\ell) j! (n-2j)!} (2x)^{n-2j} \\
&=& \f{(\ell+2n-2)(\ell+2n-4) \cd\dots\cd \ell}{n!} \, x^n  \\ 
& &  - \; \f{(\ell+2n-4)(\ell+2n-6) \cd\dots\cd \ell}{(n-2)!} \, x^{n-2} \pm \dots
%&\sr{\ell=1}=& \1{2^n} \cd \sum_{j=0}^{\lfloor\2n\rfloor}
%\f {(-1)^j(2n-2j)!}{(n-j)!j!(n-2j)!} x^{n-2j}
%\eq  \1{2^n} {2n \choose n} x^n + \dots
\end{eqnarray*}

The Legendre-Gegenbauer polynomials with small index are shown in the following diagram
\def\doarE#1#2{
\arrow <1mm> [.25,.75] from 10  01 to 20 01
\arrow <1mm> [.25,.75] from 20 -01 to 10 -01
\put {\tiny $#1 X + RD$} [b] at 15 02
\put {\tiny $#2 X - RD$} [t] at 15 -02
}
\def\doarN#1#2#3{
\arrow <1mm> [.25,.75] from -01 10 to -01 20
\arrow <1mm> [.25,.75] from  01 20 to  01 10
\put {\tiny $ XD    #1     $} [r] <0mm, 2mm> at -02 15
\put {\tiny $ R(XD  #2) #3 $} [l] <0mm,-2mm> at +02 15
}
\def\doarX#1#2{ }
\begin{eqnarray}
\bq\bq\bq\bq\bq\bq
\beginpicture
\label{Legendre spaces grid global}
\def\SC #1 #2 {\setcoordinatesystem units <.9mm,.8mm> point at {-#1} {-#2} }
\def\y#1{  \put {$\stC \langle #1 \rangle$} [] at  00  00 }
\def\yl#1{ \put {$\stC \langle #1        $} [] at -01  02 }
\def\yr#1{ \put {$\stC         #1 \rangle$} [] at +01 -02 }
\SC 00 00
\setplotarea x from 00 to 160, y from 00 to 210
\SC  00  00 \doarE{-1}{ 1} \doarN{- 1}{  }{   } \y{1}
\SC  30  00 \doarE{ 0}{ 2} \doarN{   }{-1}{   } \y{x}
\SC  60  00 \doarE{ 1}{ 3} \doarN{+ 1}{-2}{   } \y{x^2-1}
\SC  90  00 \doarE{ 2}{ 4} \doarN{+ 2}{-3}{   } \y{3x^3-3x}
\SC 120  00 \doarE{ 3}{ 5} \doarN{+ 3}{-4}{   } \yl{5x^4-}\yr{6x^2+1}
\SC 150  00 \doarX{ 4}{ 6} \doarN{+ 4}{-5}{   } \yl{7x^5-}\yr{10x^3+3x}

\SC  00  30 \doarE{ 1}{ 1} \doarN{+ 1}{  }{+ 2} \y{1}
\SC  30  30 \doarE{ 2}{ 2} \doarN{+ 2}{-1}{+ 2} \y{x}
\SC  60  30 \doarE{ 3}{ 3} \doarN{+ 3}{-2}{+ 2} \y{3x^2-1}
\SC  90  30 \doarE{ 4}{ 4} \doarN{+ 4}{-3}{+ 2} \y{5x^3-3x}
\SC 120  30 \doarE{ 5}{ 5} \doarN{+ 5}{-4}{+ 2} \yl{35x^4-}\yr{30x^2+3}
\SC 150  30 \doarX{ 6}{ 6} \doarN{+ 6}{-5}{+ 2} \yl{63x^5-}\yr{70x^3+15x}

\SC  00  60 \doarE{ 3}{ 1} \doarN{+ 3}{  }{+ 4} \y{1}
\SC  30  60 \doarE{ 4}{ 2} \doarN{+ 4}{-1}{+ 4} \y{x}
\SC  60  60 \doarE{ 5}{ 3} \doarN{+ 5}{-2}{+ 4} \y{5x^2-1}
\SC  90  60 \doarE{ 6}{ 4} \doarN{+ 6}{-3}{+ 4} \y{7x^3-3x}
\SC 120  60 \doarE{ 7}{ 5} \doarN{+ 7}{-4}{+ 4} \yl{21x^4-}\yr{14x^2+1}
\SC 150  60 \doarX{ 8}{ 6} \doarN{+ 8}{-5}{+ 4} \yl{231x^5-}\yr{210x^3+35x}

\SC  00  90 \doarE{ 5}{ 1} \doarN{+ 5}{  }{+ 6} \y{1}
\SC  30  90 \doarE{ 6}{ 2} \doarN{+ 6}{-1}{+ 6} \y{x}
\SC  60  90 \doarE{ 7}{ 3} \doarN{+ 7}{-2}{+ 6} \y{7x^2-1}
\SC  90  90 \doarE{ 8}{ 4} \doarN{+ 8}{-3}{+ 6} \y{9x^3-3x}
\SC 120  90 \doarE{ 9}{ 5} \doarN{+ 9}{-4}{+ 6} \yl{33x^4-}\yr{18x^2+1}
\SC 150  90 \doarX{10}{ 6} \doarN{+10}{-5}{+ 6} \yl{143x^5-}\yr{110x^3+15x}

\SC  00 120 \doarE{ 7}{ 1} \doarN{+ 7}{  }{+ 8} \y{1}
\SC  30 120 \doarE{ 8}{ 2} \doarN{+ 8}{-1}{+ 8} \y{x}
\SC  60 120 \doarE{ 9}{ 3} \doarN{+ 9}{-2}{+ 8} \y{9x^2-1}
\SC  90 120 \doarE{10}{ 4} \doarN{+10}{-3}{+ 8} \y{11x^3-3x}
\SC 120 120 \doarE{11}{ 5} \doarN{+11}{-4}{+ 8} \yl{143x^4-}\yr{66x^2+3}
\SC 150 120 \doarX{12}{ 6} \doarN{+12}{-5}{+ 8} \yl{39x^5-}\yr{26x^3+3x}

\SC  00 150 \doarE{ 9}{ 1} \doarN{+ 9}{  }{+10} \y{1}
\SC  30 150 \doarE{10}{ 2} \doarN{+10}{-1}{+10} \y{x}
\SC  60 150 \doarE{11}{ 3} \doarN{+11}{-2}{+10} \y{11x^2-1}
\SC  90 150 \doarE{12}{ 4} \doarN{+12}{-3}{+10} \y{13x^3-3x}
\SC 120 150 \doarE{13}{ 5} \doarN{+13}{-4}{+10} \yl{65x^4-}\yr{26x^2+1}
\SC 150 150 \doarX{14}{ 6} \doarN{+14}{-5}{+10} \yl{17x^5-}\yr{10x^3+x}

\SC  00 180 \doarE{11}{ 1} \doarN{+11}{  }{+12} \y{1}
\SC  30 180 \doarE{12}{ 2} \doarN{+12}{-1}{+12} \y{x}
\SC  60 180 \doarE{13}{ 3} \doarN{+13}{-2}{+12} \y{13x^2-1}
\SC  90 180 \doarE{14}{ 4} \doarN{+14}{-3}{+12} \y{15x^3-3x}
\SC 120 180 \doarE{15}{ 5} \doarN{+15}{-4}{+12} \yl{85x^4-}\yr{30x^2+1}
\SC 150 180 \doarX{16}{ 6} \doarN{+16}{-5}{+12} \yl{323x^5-}\yr{170x^3+15x}

\SC  00 210 \doarE{13}{ 1} \doarX{0}{} \y{1}
\SC  30 210 \doarE{14}{ 2} \doarX{1}{} \y{x}
\SC  60 210 \doarE{15}{ 3} \doarX{2}{} \y{15x^2-1}
\SC  90 210 \doarE{16}{ 4} \doarX{3}{} \y{17x^3-3x}
\SC 120 210 \doarE{17}{ 5} \doarX{4}{} \yl{323x^4-}\yr{102x^2+3}
\SC 150 210 \doarX{18}{ 6} \doarX{5}{} \yl{399x^5-}\yr{190x^3+15x}

\SC  00  00 \put { $\stC n  = 0 $ } [t] at 00 -06
\SC  30  00 \put { $\stC n  = 1 $ } [t] at 00 -06
\SC  60  00 \put { $\stC n  = 2 $ } [t] at 00 -06
\SC  90  00 \put { $\stC n  = 3 $ } [t] at 00 -06
\SC 120  00 \put { $\stC n  = 4 $ } [t] at 00 -06
\SC 150  00 \put { $\stC n  = 5 $ } [t] at 00 -06

\SC  00  00  \put { $\stC \ell = -1$ } [r] at -03 -00
\SC  00  30  \put { $\stC \ell =  1$ } [r] at -03 -00 % Legendre
\SC  00  60  \put { $\stC \ell =  3$ } [r] at -03 -00
\SC  00  90  \put { $\stC \ell =  5$ } [r] at -03 -00
\SC  00 120  \put { $\stC \ell =  7$ } [r] at -03 -00
\SC  00 150  \put { $\stC \ell =  9$ } [r] at -03 -00
\SC  00 180  \put { $\stC \ell = 11$ } [r] at -03 -00
\SC  00 210  \put { $\stC \ell = 13$ } [r] at -03 -00
\endpicture
\end{eqnarray}

Now taking into account the norming factors for the Gegenbauer polynomials the local diagram (\ref{Legendre spaces grid local})
has to be modified

\def\plotSA #1 #2 #3 #4 /{ \arrow <1mm> [.25,.75] from #1 #2 to #3 #4 }
\begin{equation}
\bq\bq\bq 
\label{Legendre polynomials grid local}
\def\SC #1 #2 {\setcoordinatesystem units <.8mm,.6mm> point at {-#1} {-#2} }
\beginpicture
\SC  00  00 \setplotarea x from -65 to +70, y from -70 to +70
\SC -60 -60 \put { $P_{n-1}^{\ell-2}$ } [] at 00 00
\SC -60  00 \put { $P_{n-1}^{\ell  }$ } [] at 00 00
\SC -60  60 \put { $P_{n-1}^{\ell+2}$ } [] at 00 00
\SC  00 -60 \put { $P_{n  }^{\ell-2}$ } [] at 00 00
\SC  00  00 \put { $P_{n  }^{\ell  }$ } [] at 00 00
\SC  00  60 \put { $P_{n  }^{\ell+2}$ } [] at 00 00
\SC +60 -60 \put { $P_{n+1}^{\ell-2}$ } [] at 00 00
\SC  60  00 \put { $P_{n+1}^{\ell  }$ } [] at 00 00
\SC +60 +60 \put { $P_{n+1}^{\ell+2}$ } [] at 00 00

% East
\SC  00  00 \plotSA +20 +01 +40 +01 / \put {$\stD \f{(n+\ell  )X + RD}{n+1} $} [b] <0mm,0mm> at +30 +03
\SC  00  60 \plotSA +20 +01 +40 +01 / \put {$\stD \f{(n+\ell+2)X + RD}{n+1} $} [b] <0mm,0mm> at +30 +03
\SC -60  60 \plotSA +20 +01 +40 +01 / \put {$\stD \f{(n+\ell+1)X + RD}{n  } $} [b] <0mm,0mm> at +30 +03
\SC -60  00 \plotSA +20 +01 +40 +01 / \put {$\stD \f{(n+\ell-1)X + RD}{n  } $} [b] <0mm,0mm> at +30 +03
\SC -60 -60 \plotSA +20 +01 +40 +01 / \put {$\stD \f{(n+\ell-3)X + RD}{n  } $} [b] <0mm,0mm> at +30 +03
\SC  00 -60 \plotSA +20 +01 +40 +01 / \put {$\stD \f{(n+\ell-2)X + RD}{n+1} $} [b] <0mm,0mm> at +30 +03

% West
\SC  00  00 \plotSA -20 -01 -40 -01 / \put {$\stD \f{ n    X-RD}{n+\ell-1}  $} [t] at -30 -03
\SC  60  00 \plotSA -20 -01 -40 -01 / \put {$\stD \f{(n+1) X-RD}{n+\ell  }  $} [t] at -30 -03
\SC  60  60 \plotSA -20 -01 -40 -01 / \put {$\stD \f{(n+1) X-RD}{n+\ell+2}  $} [t] at -30 -03
\SC  00  60 \plotSA -20 -01 -40 -01 / \put {$\stD \f{ n    X-RD}{n+\ell+1}  $} [t] at -30 -03
\SC  00 -60 \plotSA -20 -01 -40 -01 / \put {$\stD \f{ n    X-RD}{n+\ell-3}  $} [t] at -30 -03
\SC  60 -60 \plotSA -20 -01 -40 -01 / \put {$\stD \f{(n+1) X-RD}{n+\ell-2}  $} [t] at -30 -03

% North
\SC  00  00 \plotSA -01 +20 -01 +40 / \put {$\stD \f{XD+n+\ell  }{\ell  } $} [r] <+1mm,+4mm> at -03 +30
\SC  60  00 \plotSA -01 +20 -01 +40 / \put {$\stD \f{XD+n+\ell+1}{\ell  } $} [r] <+1mm,+4mm> at -03 +30
\SC -60  00 \plotSA -01 +20 -01 +40 / \put {$\stD \f{XD+n+\ell-1}{\ell  } $} [r] <+1mm,+4mm> at -03 +30
\SC -60 -60 \plotSA -01 +20 -01 +40 / \put {$\stD \f{XD+n+\ell-3}{\ell-2} $} [r] <+1mm,+4mm> at -03 +30
\SC  00 -60 \plotSA -01 +20 -01 +40 / \put {$\stD \f{XD+n+\ell-2}{\ell-2} $} [r] <+1mm,+4mm> at -03 +30
\SC  60 -60 \plotSA -01 +20 -01 +40 / \put {$\stD \f{XD+n+\ell-1}{\ell-2} $} [r] <+1mm,+4mm> at -03 +30

% South
\SC -60  00 \plotSA +01 -20 +01 -40 / \put {$\stD \f{(\ell-2)[R(\!XD - n+1) + \ell-1]}{(n+\ell-3)(n+\ell-2)} $} [l] <-1mm,-4mm> at +03 -30
\SC  00  00 \plotSA +01 -20 +01 -40 / \put {$\stD \f{(\ell-2)[R(\!XD - n  ) + \ell-1]}{(n+\ell-2)(n+\ell-1)} $} [l] <-1mm,-4mm> at +03 -30
\SC  60  00 \plotSA +01 -20 +01 -40 / \put {$\stD \f{(\ell-2)[R(\!XD - n-1) + \ell-1]}{(n+\ell-1)(n+\ell  )} $} [l] <-1mm,-4mm> at +03 -30
\SC -60  60 \plotSA +01 -20 +01 -40 / \put {$\stD \f{ \ell   [R(\!XD - n+1) + \ell+1]}{(n+\ell-1)(n+\ell  )} $} [l] <-1mm,-4mm> at +03 -30
\SC  00  60 \plotSA +01 -20 +01 -40 / \put {$\stD \f{ \ell   [R(\!XD - n  ) + \ell+1]}{(n+\ell  )(n+\ell+1)} $} [l] <-1mm,-4mm> at +03 -30
\SC  60  60 \plotSA +01 -20 +01 -40 / \put {$\stD \f{ \ell   [R(\!XD - n-1) + \ell+1]}{(n+\ell+1)(n+\ell+2)} $} [l] <-1mm,-4mm> at +03 -30

% Northwest
\SC  00  00  \plotSA -19.5 +20.5 -39.5 +40.5 / \put {$\stD \f D \ell    $} [lt] at -35  40
\SC  60 -60  \plotSA -19.5 +20.5 -39.5 +40.5 / \put {$\stD \f D{\ell-2} $} [lt] at -35  40

% Southeast
\SC  00 00 \plotSA +19.5 -20.5 +39.5 -40.5 / \put {$\stD \f{(\ell-2)[(\ell-1)X+RD]}{(n+1)(n+\ell-1)} $} [rb] <4mm,-5mm> at +35 -40
\SC -60 60 \plotSA +19.5 -20.5 +39.5 -40.5 / \put {$\stD \f{ \ell   [(\ell+1)X+RD]}{ n   (n+\ell  )} $} [rb] <4mm,-5mm> at +35 -40

% Northeast
%\SC  00  00  \plotSA +19.5 +20.5 +39.5 +40.5 / \put {$\stD          $} [rt] at +35  40 % \plotNoEa{n+\ell}
%\SC -60 -60  \plotSA +19.5 +20.5 +39.5 +40.5 / \put {$\stD          $} [rt] at +35  40 % \plotNoEa{n+\ell-2}

% Southwest
%\SC  00  00 \plotSA -19.5 -20.5 -39.5 -40.5 / \put {$\stD          $} [lb] at -35 -40 % \plotSoWe{n}{\ell-1}
%\SC  60  60 \plotSA -19.5 -20.5 -39.5 -40.5 / \put {$\stD          $} [lb] at -35 -40 % \plotSoWe{n}{\ell+1}
\endpicture
\end{equation}

\mip
The norm factors were not incorporated in the first Legendre-Gegenbauer diagram (\ref{Legendre spaces grid local}), because
then the scalar commutator conditions (\ref{commutator condition Laguerre horizontal}), (\ref{commutator condition Laguerre vertical}) and (\ref{commutator condition Laguerre diagonal}) are destroyed.
We can read off some of the classical identities for Legendre-Gegenbauer polynomials from the diagram
(\ref{Legendre polynomials grid local}).

\begin{theorem}[Legendre-Gegenbauer polynomials]
For the Legendre-Gegenbauer polynomials the following relations hold
\begin{itemize}
\item[(i)] Three term recurrence relation.
\begin{eqnarray*}
( n + 1 ) P_{n+1}^\ell + ( n + \ell - 1 ) P_{n-1}^\ell &=& (2n+\ell)X P_n^\ell.
\end{eqnarray*}
\item[(ii)] There are various so called three-point-rules
\begin{eqnarray*}
%\bq\bq\bq\bq \begin{array}[c]{ r c l c l } 
( n+\ell) P_n^\ell 
&=& ( XD + n + \ell ) P_n^\ell - XD P_n^\ell \\ 
&=&  \ell P_n^{\ell+2} -\ell X P_{n-1}^{\ell+2} \\[1.5ex] 
    n X P_n^\ell     
&=& ( n X - R D ) P_n^\ell + R D P_n^\ell \\   
&=& ( n + \ell-1 ) P_{n-1}^\ell + \ell R P_{n-1}^{\ell+2}  \\[1.5ex]
    ( n+\ell) P_n^\ell 
&=& ( n+\ell + R D ) P_n^\ell - R D P_n^\ell 
\\&=& (n+1) P_{n+1}^{\ell} - \ell R P_{n-1}^{\ell+2} \\[1.5ex]
    ( \ell - 1 ) ( \ell - 2) X P_n^\ell 
&=& ( \ell - 2) [ ( \ell - 1 ) X + R D ] P_n^\ell - (\ell-2) R D P_n^\ell \\ 
&=& (n+1) (n+\ell-1) P_{n+1}^{\ell-2} - \ell(\ell-2) R P_{n-1}^{\ell+2}
%\end{array}
\end{eqnarray*}
\item[(iii)] Rodrigues formula.
For an index pair $(n,\ell) \in \N \x (2\N +1)$ the Legendre-Gegenbauer polynomial $P_n^\ell$ is given by
\begin{eqnarray*}
       P_n^\ell
&\eq& \f{(\ell+2n-2) \cd (\ell+2n-4) \cd \dots \cd \ell} {(\ell+2n-1) \cd \dots \cd (\ell+n) \cd n!} \;
                                                                     R^{-\2{\ell-1}}\, D^n R^{n+\2{\ell-1}} \, 1
\end{eqnarray*}
\end{itemize}
\end{theorem}

\begin{proof} \
(i) By looking at the two horizontal arrows in East and West direction starting at $P_n^\ell$
in (\ref{Legendre polynomials grid local}) we get the two equations
\begin{eqnarray*}
( n + 1        ) P_{n+1}^\ell &=& [(n+\ell) X + RD ) P_n^\ell \\
( n + \ell - 1 ) P_{n-1}^\ell &=& [ n       X - RD ) P_n^\ell .
\end{eqnarray*}
Adding up the two equations yields the three term recurrence relation.
\mip
(iii) (1) We are going to show the two following equalities
\begin{eqnarray}
\nn      P_n^\ell
&\eq& \f{\ell \cd(RD + (\ell+1)X)}{n(n+\ell)} \cd \f{(\ell+2) \cd(RD + (\ell+3)X)}{(n-1)(n+\ell+1)} \cd \dots  \\
&   & \hs{1cm} \dots \cd \f{(\ell+2n-2) \cd (RD + (\ell+2n-1)X)}{1 \cd (2n+\ell-1)} \, 1 \\
\label{three equations}
&\eq& \f{(\ell+2n-2) \cd (\ell+2n-4) \cd \dots \cd \ell} {(\ell+2n-1)!} \, D^{n+\ell-1} \, R^{n+\2{\ell-1}} \; 1 \\
\nn
&\eq& \f{(\ell+2n-2) \cd (\ell+2n-4) \cd \dots \cd \ell} {(\ell+2n-1) \cd \dots \cd (\ell+n) \cd n!} \,
                                                                     R^{-\2{\ell-1}}\, D^n R^{n+\2{\ell-1}} \, 1
\end{eqnarray}
(2) The first equation is clear, since the operator on the right hand side belongs to the $n$ step southeast diagonal path
\begin{eqnarray*}
&& 1  \eq P_0^{\ell+2n} \q\llongrightarrow\q P_1^{\ell+2n-2} \q\llongrightarrow\q \dots \q\llongrightarrow\q P_n^{\ell}
\end{eqnarray*}
See the diagrams (\ref{Legendre spaces grid global}) and (\ref{Legendre polynomials grid local}).
\mip
(3) We prove the second equation for $\ell = 1$. It is valid within the Weyl algebra, independent of the representation.
Since this is clear for the scalar factors it remains to prove by induction over $n \in\N$ that
\begin{eqnarray*}
& & (RD + 2X) \cd (RD + 4X) \cd \dots \cd (RD + 2nX) \eq D^n R^n
\end{eqnarray*}
For $n=1$ we have
\begin{eqnarray*}
RD + 2X &\eq& DR.
\end{eqnarray*}
and then
\begin{eqnarray*}
&   & (RD + 2X) \cd (RD + 4X) \cd \dots \cd \cd (RD + 2nX) \cd (RD + 2(n+1)X) \\
&\eq& D^n R^n \cd (RD + 2(n+1)X) \\
&\eq& D^n ( R^{n+1}D + 2(n+1) X R^n) \\
&\eq& D^{n+1} R^{n+1}
\end{eqnarray*}

(4) We now perform the induction step $(n,\ell) \to (n-1,\ell+2)$. Here we have
\begin{eqnarray*}
    P_{n-1}^{\ell+2}
\eq \f D\ell \, P_n^\ell
&=& \f D\ell \f{(\ell+2n-2)(\ell+2n-4) \cd \dots \cd \ell}{(\ell+2n-1)!} D^{n+\ell-1} R^{\2{\ell+2n-1}} \, 1 \\
&=&          \f{(\ell+2n-2)(\ell+2n-4) \cd \dots \cd (\ell+2)}{(\ell+2n-1)!} D^{n+\ell} R^{\2{\ell+2n-1}}\, 1.
\end{eqnarray*}
So far we have proved the two upper lines in (\ref{three equations}) for all $n\in\N$, $\ell \in 2\N+1$.
\mip
(5) We prove the third identity in (\ref{three equations}), it is
\begin{eqnarray*}
\label{Legendre Rodrigues 1}
\1{(n+\ell-1)!} R^{\2{\ell-1}} D^{n+\ell-1} R^{n+\2{\ell-1}} \, 1 &\eq& \1{n!} D^n R^{n+\2{\ell-1}} \, 1.
\end{eqnarray*}
This is an equation between two polynomials, we check it by applying the derivative operator $r$ times for all $r = 0,\dots,n$
\begin{eqnarray*}
\label{Legendre Rodrigues 1}
\1{(n+\ell-1)!} D^r R^{\2{\ell-1}} D^{n+\ell-1} R^{n+\2{\ell-1}} \, 1 &\eq& \1{n!} D^{n+r} R^{n+\2{\ell-1}} \, 1.
\end{eqnarray*}
and then comparing the leading coefficient. In fact it is
\begin{eqnarray*}
\f{(2n+\ell-1)!}{n! \cd (n+\ell-1-r)!}
\end{eqnarray*}
on both sides. So our claim follows.
\end{proof}

\section{The $h$-Weyl algebra}
Now for fixed $h \not= 0$ we consider the so called \emph{$h$-Weyl algebra}. It is generated as an associative unital $\C$ algebra by three operators $M,D,X$ with the four relations
\begin{equation}\label{rel h Weyl algebra}
[M,D] = 0, \q
[M,X] = h^2 D, \q
[D,X] = M, \q
M^2 - h^2 D^2 = 1.
\end{equation}
It is not hard to derive the following additional relation
\begin{eqnarray}
\label{hWeyl q} D X M - M X D &=& 1
\end{eqnarray}
\mip
The $h$--Weyl algebra has two particular representations. The first one, called diff representation is
on a space of functions (or appropriate subspace) with a discrete variable
\begin{eqnarray}\label{Diff representation hWeyl algebra}
\cF( h\Z,\C), \q\q\q
\left\{
\begin{array}[c]{ r  c  l }
M f(x) &\defeq& \2{f(x+h) + f(x-h)}
\\[1.3ex]
D f(x) &\defeq& \f{f(x+h) - f(x-h)}{2h}
\\[1.3ex]
X f(x) &\defeq& x \cd f(x).
\end{array}
 \right.
\end{eqnarray}
The symbols $M$ (mix), $D$ (difference) and $X$ arise from this difference operator  representation.
The second representation, called trig representation, is related to the first one by Fourier transform and Pontryagin duality. The $h$--Weyl algebra
acts on smooth functions on a circle with radius $\1h$ ( $= \f{2\pi}h$--periodic functions on $\R$) (or appropriate subspace)
\begin{eqnarray}\label{Trig representation hWeyl algebra}
\cC^\8({\stB \1h}\S,\C), \q\q\q
\left\{
\begin{array}[c]{ r  c  l }
M f(x) &\defeq& \cos(h x) f(x)
\\[1.3ex]
D f(x) &\defeq& \f{i \sin(hx)}{h} f(x)
\\[1.3ex]
 X f(x) &\defeq& i f'(x).
\end{array}
\right.
\end{eqnarray}
For $M = I$ and $h = 0$ the $h$--Weyl algebra reduces to the original Weyl algebra. The two representations reduce to the standard representation of the Weyl algebra.

\section{The Binomial grid}

From now on we will only consider the $h$-Weyl algebra with $h=1$.
We define the following operators
\begin{eqnarray*}
\label{def Gj} G_j &\defeq& j D +  MX \eq (j+1) D +     XM \\
\label{def Lj} L_j &\defeq& j M +  DX \eq (j+1) M +  XD
\end{eqnarray*}
and note some relations
\begin{eqnarray}
\label{ToR 1} L_j D - D L_{j-1} 
&=& (j M +  DX) D - D (j M +  XD)  \\ 
\nn &=& 0 \\[1.5ex] 
\label{ToR 2} G_j M - M G_{j-1}         
&=& (j D + MX) M - M (j D + XM)  \\ 
\nn &=& 0 \\[1.5ex]
\nn           G_{j-1} L_j - L_{j-1} G_j 
&=& (j D + XM) (j M +  DX) - (j M +  XD)(j D + MX)       \\
\label{ToR 3}                           
&=& j X ( M^2 -  D^2) + j (  D^2 - M^2) X \\ 
\nn &=& 0 \\[1.5ex]
\nn       L_{j-1} L_j -  G_{j-1} G_j    
&=& (j M +  XD)(j M +  DX) -  (j D + XM)(j D + MX) \\
\nn                                     
\nn &=& j^2 (M^2 -D^2) + X ( D^2 -M^2) X  \\
\nn & & \q + \, j  (MDX -DMX + XDM - XMD) \\ 
\label{ToR 4}  &=& j^2 - X^2
\end{eqnarray}
Define for $n,m \in \Z$ the (second order difference) \emph{binomials operators}
\begin{eqnarray}
\label{definition Rnm}
H_{nm} \defeq H_{nm}^\7{bin}
       &\defeq&               (n+m+1)  M^2    +  DXM - (n+1) \\
\nn    &\sr{(\ref{hWeyl q})}=& (n+m+1)  M^2    +  MXD -  n    \\
\nn    &=&                    (n+m+1) (M^2-1) +  MXD + (m+1) \\
\nn    &=&                    (n+m+1)  D^2    +  MXD + (m+1) \\
\nn    &\sr{(\ref{hWeyl q})}=& (n+m+1)  D^2    +  DXM +  m
\end{eqnarray}
in the $1$-Weyl algebra and then the spaces
\begin{eqnarray*}
U_{nm} \defeq \ker H_{nm}
\end{eqnarray*}
with respect to any representation. The following diagram shows the local structure of the binomial grid.
\begin{equation}
\label{binomial grid local}
\def\SC #1 #2 {\setcoordinatesystem units <.5mm,.5mm> point at {-#1} {-#2} }
\def\plotSA #1 #2 #3 #4 /{ \arrow <1mm> [.25,.75] from #1 #2 to #3 #4 }
\beginpicture
\SC  00  00 \setplotarea x from -70 to +70, y from -70 to +70

\SC -60 -60 \put { $U_{n-1,m-1}$} [] at 00 00
\SC -60  00 \put { $U_{n-1,m  }$} [] at 00 00
\SC -60 +60 \put { $U_{n-1,m+1}$} [] at 00 00
\SC  00 -60 \put { $U_{n  ,m-1}$} [] at 00 00
\SC  00  00 \put { $U_{n  ,m  }$} [] at 00 00
\SC  00 +60 \put { $U_{n  ,m+1}$} [] at 00 00
\SC +60 -60 \put { $U_{n+1,m-1}$} [] at 00 00
\SC +60  00 \put { $U_{n+1,m  }$} [] at 00 00
\SC +60 +60 \put { $U_{n+1,m+1}$} [] at 00 00

% East
\SC  00  00 \plotSA +20  +01  +40  +01 / \put {$\stC M  $} [b] at +30  02
\SC  00  60 \plotSA +20  +01  +40  +01 / \put {$\stC M  $} [b] at +30  02
\SC -60  60 \plotSA +20  +01  +40  +01 / \put {$\stC M  $} [b] at +30  02
\SC -60 -60 \plotSA +20  +01  +40  +01 / \put {$\stC M  $} [b] at +30  02
\SC -60  00 \plotSA +20  +01  +40  +01 / \put {$\stC M  $} [b] at +30  02
\SC  00 -60 \plotSA +20  +01  +40  +01 / \put {$\stC M  $} [b] at +30  02

% West
\SC  00  00 \plotSA -20  -01  -40  -01 / \put {$\stC L_{n+m  } $} [t] at -30 -02
\SC  60  00 \plotSA -20  -01  -40  -01 / \put {$\stC L_{n+m+1} $} [t] at -30 -02
\SC  60  60 \plotSA -20  -01  -40  -01 / \put {$\stC L_{n+m+2} $} [t] at -30 -02
\SC  00  60 \plotSA -20  -01  -40  -01 / \put {$\stC L_{n+m+1} $} [t] at -30 -02
\SC  00 -60 \plotSA -20  -01  -40  -01 / \put {$\stC L_{n+m-1} $} [t] at -30 -02
\SC  60 -60 \plotSA -20  -01  -40  -01 / \put {$\stC L_{n+m  } $} [t] at -30 -02

% North
\SC  00  00 \plotSA -01  +20  -01  +40 / \put {$\stC - D $} [r] at -03 +30
\SC  60  00 \plotSA -01  +20  -01  +40 / \put {$\stC - D $} [r] at -03 +30
\SC -60  00 \plotSA -01  +20  -01  +40 / \put {$\stC - D $} [r] at -03 +30
\SC -60 -60 \plotSA -01  +20  -01  +40 / \put {$\stC - D $} [r] at -03 +30
\SC  00 -60 \plotSA -01  +20  -01  +40 / \put {$\stC - D $} [r] at -03 +30
\SC  60 -60 \plotSA -01  +20  -01  +40 / \put {$\stC - D $} [r] at -03 +30

% South
\SC  00  00 \plotSA +01  -20  +01  -40 / \put {$\stC G_{n+m  } $} [l] at +03 -30
\SC  60  00 \plotSA +01  -20  +01  -40 / \put {$\stC G_{n+m+1} $} [l] at +03 -30
\SC  60  60 \plotSA +01  -20  +01  -40 / \put {$\stC G_{n+m+2} $} [l] at +03 -30
\SC  00  60 \plotSA +01  -20  +01  -40 / \put {$\stC G_{n+m+1} $} [l] at +03 -30
\SC -60  60 \plotSA +01  -20  +01  -40 / \put {$\stC G_{n+m  } $} [l] at +03 -30
\SC -60  00 \plotSA +01  -20  +01  -40 / \put {$\stC G_{n+m-1} $} [l] at +03 -30
\endpicture
\end{equation}

\begin{theorem}[Binomial grid]
The Binomial grid is a well defined SIE grid.
The various circuit endomorphisms $U_{nm} \to U_{nm}$ act as scalars as follows
\begin{eqnarray*}
\begin{array}[c]{ l @{\q} r @{\;\;=\;\;}  l }
\mbox{$\rightleftharpoons$ East}  & L_{n+m+1} M     & n+1 \\[1.5ex]
\mbox{$\rightleftharpoons$ West}  & M L_{n+m}       & n   \\[1.5ex]
\mbox{$\rightleftharpoons$ North} & G_{n+m+1} (- D) & m+1 \\[1.5ex]
\mbox{$\rightleftharpoons$ South} & (- D) G_{n+m  } & m   \\[1.5ex]
\mbox{$\Box \circlearrowleft $} & G_{n+m+1} \, L_{n+m+2} \, (-D) \, M & (n+1) (m+1) \\[1.5ex]
\mbox{$\Box \circlearrowright$} & L_{n+m+1} \, G_{n+m+2} \, M \, (-D) & (n+1) (m+1)
\end{array}
\end{eqnarray*}
\end{theorem}

\begin{proof}
% (0) First note the following equivalent expressions for the operator $H_{nm}$:
(1) The horizontal narrow circuit endomorphisms are
\begin{eqnarray*}
A_{nm}^\Rq &=& L_{n+m+1} M \eq (n+m+1) M^2 +  DXM \eq  H_{nm} + (n+1) \\[1.5ex]
A_{nm}^\Lq &=& M L_{n+m}   \eq (n+m+1) M^2 +  MXD \eq  H_{nm} +  n
\end{eqnarray*}

When choosing the number sequence $\aa_{nm} = n$, we find that the commutator condition (\ref{commutator condition}) in
Lemma \ref{Lemma: Scalar commutators} is fulfilled,
\begin{eqnarray}
\label{commutator condition binomials horizontal}
A_{nm}^\Rq - A_{nm}^\Lq &\eq& 1 \eq \aa_{n+1,m} - \aa_{nm}.
\end{eqnarray}
Lemma \ref{Lemma: Scalar commutators} shows that the horizontal ladders
\begin{equation}
\label{Binomials horizontal ladder}
\beginpicture
\SC  00 00 \setplotarea x from -55 to +55, y from -03 to 03 
\SC -50 00 \put {$\dots$}   [] at 00 00
\SC +50 00 \put {$\dots$}   [] at 00 00
\SC -30 00 \put {$U_{n-1,m}$} [] <0mm,0mm> at -05 00
\SC  00 00 \put {$U_{n   m}$} [] at 00 00
\SC +30 00 \put {$U_{n+1,m}$} [] <0mm,0mm> at +05 00
\SC -15 00 \arrow <1mm> [.25,.75] from -05 01 to +05  01 \put {\tiny $M$} [b] at 00  02
           \arrow <1mm> [.25,.75] from +05 00 to -05  00 \put {\tiny $L_{n+m}$} [t] at 00 -01
\SC +15 00 \arrow <1mm> [.25,.75] from -05 01 to +05  01 \put {\tiny $M$} [b] at 00  02
           \arrow <1mm> [.25,.75] from +05 00 to -05  00 \put {\tiny $L_{n+m+1}$} [t] at 00 -01
\endpicture
\end{equation}
in the binomial grid (\ref{binomial grid local}) are well defined SIE subladders.
Within the diff representation (\ref{Diff representation hWeyl algebra}) the operator $M$ is the discrete heat distribution
operator. So the above ladder can be called the Heat ladder.
\mip
(2) The vertical circuit endomorphisms in North and South direction are:
\begin{eqnarray*}
B_{n\ell}^\Rq &=& G_{n+m+1} (- D) \eq - (n+m+1)  D^2 -  MXD \\ 
&=& - H_{nm} + (m+1) \\
B_{n\ell}^\Lq &=& (- D)   G_{n+m} \eq - (n+m+1)  D^2 -  DXM \\ 
&=& - H_{nm} +  m
\end{eqnarray*}
The difference of the two circuit endomorphisms is again a scalar,
\begin{eqnarray}
\label{commutator condition binomials horizontal}
B_{nm}^\Rq - B_{nm}^\Lq &\eq& 1 \eq \ba_{n,m+1} - \ba_{nm},
\end{eqnarray}
where $\ba_{nm} = m$.
Lemma \ref{Lemma: Scalar commutators} again implies that the vertical ladders
\begin{equation}
\label{Binomials vertical ladder}
\beginpicture
\SC  00 00 \setplotarea x from -55 to +55, y from -03 to 03
\SC -50 00 \put {$\dots$}   [] at 00 00
\SC +50 00 \put {$\dots$}   [] at 00 00
\SC -30 00 \put {$U_{n,m-1}$} [] <0mm,0mm> at -05 00
\SC  00 00 \put {$U_{n m  }$} [] at 00 00
\SC +30 00 \put {$U_{n,m+1}$} [] <0mm,0mm> at +05 00
\SC -15 00 \arrow <1mm> [.25,.75] from -05 01 to +05  01 \put {\tiny $-D$} [b] at 00  02
           \arrow <1mm> [.25,.75] from +05 00 to -05  00 \put {\tiny $G_{n+m}$} [t] at 00 -01
\SC +15 00 \arrow <1mm> [.25,.75] from -05 01 to +05  01 \put {\tiny $-D$} [b] at 00  02
           \arrow <1mm> [.25,.75] from +05 00 to -05  00 \put {\tiny $G_{n+m+1}$} [t] at 00 -01
\endpicture
\end{equation}

in the binomial grid (\ref{binomial grid local}) are well defined.

\mip
(3) We use (\ref{ToR 1}), (\ref{ToR 2}) and steps (1) and (2) in order to compute the Northeast square circuit
endomorphism on $U_{nm}$ in counterclockwise and clockwise direction
\begin{eqnarray*}
\bq\bq G_{n+m+1} \, L_{n+m+2} \, (-D)      \, M
&=& G_{n+m+1} \, (-D)      \, L_{n+m+1} \, M    \eq (n+1) (m+1) \\[1.5ex]
\bq\bq      L_{n+m+1} \, G_{n+m+2} \,  M        \, (-D)
&=& L_{n+m+1} \, M         \, G_{n+m+1} \, (-D) \eq (n+1) (m+1)
\end{eqnarray*}
The proof is finished.
\end{proof}

\section{The Binomials}
The relations
\begin{eqnarray*}
G_0 &=& MX \\
L_0 &=& DX \\
X   &=& (M+D)(M-D) X \eq (M+D) (G_0 - L_0 ) \\
H_{00} &=& M L_0 \eq  M D X
\end{eqnarray*}
see (\ref{rel h Weyl algebra}), show that
\begin{eqnarray}
\label{def W00}
W_{00} \defeq \ker X \eq \ker G_0 \cap \ker L_0 \; \subseteq \ker H_{00} \eq U_{00}.
\end{eqnarray}
This subspace of $U_{00}$ generates an SIE subrepresentation $(W_{nm})$ of the grid $(U_{nm})$ by
\begin{eqnarray*}
W_{nm} \defeq M^n D^m (W_{00}).
\end{eqnarray*}

Within the trigonometric representation (\ref{Trig representation hWeyl algebra})
we have $W_{00}  = \ker X = \lgl 1 \rgl $, where $1$ is the constant--$1$ function on $\S$.
So the subrepresentation $(W_{nm})$ is given by
\begin{eqnarray*}
W_{nm} &=& \lgl \cos^{n}(x) \; \sin^m (x) \rgl.
\end{eqnarray*}
Within the Diff representation (\ref{Diff representation hWeyl algebra}) we have $W_{00} = \ker X = \lgl \da_0 \rgl $,
where $\da_z$ is the ``Kronecker Delta'' function on $\Z$.
The representation spaces $(W_{nm})$ contain centered binomial functions with alternating signs.
The horizontal ladder $(W_{n0})$ is the ladder of the classical centered binomials.
They are related to the discrete Harmonic Oscillator \cite{HilgerHOE}.
This representation $(W_{nm})$ is displayed in the following last diagram
\[
\bq\bq\bq\bq\bq\bq\bq
\beginpicture
\def\doarW#1{
\arrow <1mm> [.25,.75] from -10 -01 to -20 -01
\arrow <1mm> [.25,.75] from -20 +01 to -10 +01
\put {\tiny $  M     $} [b] at -15  02
\put {\tiny $ L_{#1} $} [t] at -15 -02
}
\def\doarS#1{
\arrow <1mm> [.25,.75] from -01 -20 to -01 -10
\arrow <1mm> [.25,.75] from  01 -10 to  01 -20
\put {\tiny $  -D    $} [r] at -02 -15
\put {\tiny $ G_{#1} $} [l] at +02 -15
}
\def\doarX#1{ }
\def\ya#1{ \put {$\stC \langle #1 \rangle$} [] at  00  00 }
\def\yb#1{ \put {$\stD \langle #1        $} [] at -01  02 }
\def\yc#1{ \put {$\stD         #1 \rangle$} [] at +01 -02 }
\def\yp#1{ \put {$\stD \langle #1        $} [] at -01  03 }
\def\yq#1{ \put {$\stD         #1        $} [] at -01  00 }
\def\yr#1{ \put {$\stD         #1 \rangle$} [] at +01 -03 }

\SC 00 00
\setplotarea x from 00 to 140, y from -10 to 120
\SC  00  00 \doarX{  }  \doarX{  }  \ya{    \da_{ 0}             }
\SC  30  00 \doarW{ 1}  \doarX{ 1}  \ya{    \da_{-1} +  \da_{ 1} }
\SC  60  00 \doarW{ 2}  \doarX{ 2}  \yb{    \da_{-2} + 2\da_{ 0} }
                                    \yc{  \q\q + \da_{ 2}             }
\SC  90  00 \doarW{ 3}  \doarX{ 3}  \yb{    \da_{-3} + 3\da_{-1} }
                                    \yc{ + 3\da_{ 1} +  \da_{ 3} }
\SC 120  00 \doarW{ 4}  \doarX{ 4}  \yb{    \da_{-4} + 4\da_{-2} }
                                    \yc{ + 6\da_{ 0} + 4\da_{ 2} + \da_{ 4} }
\SC  00  30 \doarX{ 1}  \doarS{ 1}  \ya{    \da_{-1} -  \da_{ 1} }
\SC  30  30 \doarW{ 2}  \doarS{ 2}  \ya{    \da_{-2} -  \da_{ 2} }
\SC  60  30 \doarW{ 3}  \doarS{ 3}  \yb{    \da_{-3} +  \da_{-1} }
                                    \yc{ -  \da_{ 1} -  \da_{ 3} }
\SC  90  30 \doarW{ 4}  \doarS{ 4}  \yb{    \da_{-4} + 2\da_{-2} }
                                    \yc{ - 2\da_{ 2} -  \da_{ 4} }
\SC 120  30 \doarW{ 5}  \doarS{ 5}  \yp{    \da_{-5} + 3\da_{-3} }
                                    \yq{ + 2\da_{-1} - 2\da_{-1} }
                                    \yr{ - 3\da_{-3} -  \da_{-5} }
\SC  00  60 \doarX{ 2}  \doarS{ 2}  \yb{    \da_{-2} - 2\da_{ 0} }
                                    \yc {  \q\q      +  \da_{ 2} }
\SC  30  60 \doarW{ 3}  \doarS{ 3}  \yb{    \da_{-3} -  \da_{-1} }
                                    \yc{ -  \da_{ 1} +  \da_{ 3} }
\SC  60  60 \doarW{ 4}  \doarS{ 4}  \yb{    \da_{-4} - 2\da_{ 0} }
                                    \yc{ \q\q         + \da_{ 4} }
\SC  90  60 \doarW{ 5}  \doarS{ 5}  \yp{    \da_{-5} +  \da_{-3} }
                                    \yq{ - 2\da_{-1} - 2\da_{ 1} }
                                    \yr{ +  \da_{ 3} +  \da_{ 5} }
\SC 120  60 \doarW{ 6}  \doarS{ 6}  \yp{    \da_{-6} + 2\da_{-4} \q }
                                    \yq{ -  \da_{-2} - 4\da_{ 0} - \da_{ 2} }
                                    \yr{ \q          + 2\da_{ 4} + \da_{ 6} }
\SC  00  90 \doarX{ 3}  \doarS{ 3}  \yb{    \da_{-3} - 3\da_{-1} }
                                    \yc{ + 3\da_{ 1} -  \da_{ 3} }
\SC  30  90 \doarW{ 4}  \doarS{ 4}  \yb{    \da_{-4} - 2\da_{-2} }
                                    \yc{ + 2\da_{ 2} -  \da_{ 4} }
\SC  60  90 \doarW{ 5}  \doarS{ 5}  \yp{    \da_{-5} -  \da_{-3} }
                                    \yq{ - 2\da_{-1} - 2\da_{ 1} }
                                    \yr{ +  \da_{ 3} -  \da_{ 5} }
\SC  90  90 \doarW{ 6}  \doarS{ 6}  \yb{    \da_{-6} - 3\da_{-2} }
                                    \yc{ + 3\da_{ 2} -  \da_{ 6} }
\SC 120  90 \doarW{ 6}  \doarS{ 6}  \yp{    \da_{-7} -  \da_{-5}  - 3\da_{-3} }
                                    \yq{ - 3\da_{-1} + 3\da_{ 1} }
                                    \yr{\q + 3\da_{ 3} - \da_5 - \da_7 }
\SC  00 120 \doarX{ 4}  \doarS{ 4}  \yb{    \da_{-4} - 3\da_{-2} }
                                    \yc{ + 6\da_{ 0} - 4\da_{ 2} + \da_4 }
\SC  30 120 \doarW{ 5}  \doarS{ 5}  \yp{    \da_{-5} - 3\da_{-3} }
                                    \yq{ + 2\da_{-1} + 2\da_{ 1} }
                                    \yr{ - 3\da_{ 3} +  \da_{ 5} }
\SC  60 120 \doarW{ 6}  \doarS{ 6}  \yp{    \da_{-6} - 2\da_{-4} \q }
                                    \yq{ -  \da_{-2} + 4\da_{ 0} - \da_2 }
                                    \yr{\q-2\da_{ 4} +  \da_{ 6} }
\SC  90 120 \doarW{ 7}  \doarS{ 7}  \yp{    \da_{-7} - \da_{-5} - 3\da_{-3} }
                                    \yq{ + 3\da_{-1} - 3\da_{ 1} }
                                    \yr{- 3\da_{ 3} - \da_{ 5} + \da_{ 7} }
\SC 120 120 \doarW{ 8}  \doarS{ 8}  \yb{    \da_{-8} - 4\da_{-4} }
                                    \yc{ + 6\da_{ 0} - 4\da_{ 4} - \da_{ 8} }

\SC  00  00 \put { $\stC n  = 0 $ } [t] at 00 -06
\SC  30  00 \put { $\stC n  = 1 $ } [t] at 00 -06
\SC  60  00 \put { $\stC n  = 2 $ } [t] at 00 -06
\SC  90  00 \put { $\stC n  = 3 $ } [t] at 00 -06
\SC 120  00 \put { $\stC n  = 4 $ } [t] at 00 -06

\SC  00  00  \put { $\stC m = 0$ } [r] at -06 -00
\SC  00  30  \put { $\stC m = 1$ } [r] at -08 -00
\SC  00  60  \put { $\stC m = 2$ } [r] at -08 -00
\SC  00  90  \put { $\stC m = 3$ } [r] at -08 -00
\SC  00 120  \put { $\stC m = 4$ } [r] at -08 -00
\endpicture
\]
There are some interesting identities for operators appearing in this subrepresentation $(W_{nm})$ of the binomial grid.

\begin{theorem}[Binomials]
Going two steps in negative direction within the horizontal ladders $m = 0,1$ or vertical ladders $n= 0,1$, respectively,
means multiplication by some quadratic polynomial. With (\ref{ToR 4}) we get
\begin{eqnarray*}
\begin{array}[c]{ l c l c l }
L_{n-1} L_{n  } &=& n^2 -X^2      &:& W_{n,0} \to W_{n-2,0} \\[1.2ex]
L_{n  } L_{n+1} &=& (n+1)^2 - X^2 &:& W_{n,1} \to W_{n-2,1} \\[1.2ex]
G_{m-1} G_{m  } &=& X^2 - m^2     &:& W_{0,m} \to W_{0,m-2} \\[1.2ex]
G_{m  } G_{m+1} &=& X^2 - (m+1)^2 &:& W_{1,m} \to W_{1,m-2}
\end{array}
\end{eqnarray*}
\end{theorem}

\begin{proof}
The relations  (\ref{def W00}) and (\ref{ToR 1}), (\ref{ToR 2}) allow to show by induction that
\begin{eqnarray*}
W_{n,0} &\subseteq& \ker G_n \q \mbox{for all } n \in \N_0 \\
W_{0,m} &\subseteq& \ker L_m \q \mbox{for all } m \in \N_0.
\end{eqnarray*}

Then the statement of the theorem follows directly with the relation (\ref{ToR 4}):
\begin{eqnarray*}
\bq\bq
\begin{array}[c]{ l @{\;}c@{\;} l @{\;}c@{\;} l @{\;}c@{\;} l }
L_{n-1} L_{n  } &=& L_{n-1} L_{n  } - G_{n-1} G_{n  } \q\! &=& n^2 -X^2      &:& W_{n,0} \to W_{n-2,0} \\[1.2ex]
L_{n  } L_{n+1} &=& L_{n  } L_{n+1} - G_{n  } G_{n+1} \q\! &=& (n+1)^2 - X^2 &:& W_{n,1} \to W_{n-2,1} \\[1.2ex]
G_{m-1} G_{m  } &=& G_{m-1} G_{m  } - L_{m-1} L_{m  }      &=& X^2 - m^2     &:& W_{0,m} \to W_{0,m-2} \\[1.2ex]
G_{m  } G_{m+1} &=& G_{m  } G_{m+1} - L_{m  } L_{m+1}      &=& X^2 - (m+1)^2 &:& W_{1,m} \to W_{1,m-2}
\end{array}
\end{eqnarray*}
\end{proof}

email: Stefan.Hilger@ku.de

\end{document}